\documentclass[12pt,reqno]{amsart}
\usepackage[T1]{fontenc}
\usepackage[utf8]{inputenc}
\usepackage[english]{babel}
\usepackage{fullpage}
\usepackage{times}
\usepackage{colonequals}
\usepackage{amsmath,amssymb,amsthm,mathtools,url}
\usepackage{xcolor}
\usepackage{comment}
\usepackage{bbm}
\usepackage{enumerate}
\usepackage{bm}
\usepackage{graphicx}
\usepackage{mathrsfs}
\usepackage[colorlinks=true,pdfstartview=FitH,linkcolor=blue,citecolor=blue,urlcolor=blue]{hyperref}
\usepackage{lineno}

\newtheorem{thm}{Theorem}[section]
\newtheorem{lem}[thm]{Lemma}
\newtheorem{prop}[thm]{Proposition}
\newtheorem{cor}[thm]{Corollary}

\theoremstyle{definition}
\newtheorem{defn}[thm]{Definition}

\newtheorem{rem}[thm]{Remark}

\newcommand{\Nm}{\operatorname{N}}
\newcommand{\Tr}{\operatorname{Tr}}
\newcommand{\F}{\mathbb{F}}
\newcommand{\cA}{\mathcal A}
\newcommand{\cD}{\mathcal D}
\newcommand{\cF}{\mathcal F}
\newcommand{\cH}{\mathcal H}

\newcommand{\cP}{\mathcal P}

\newcommand{\cM}{\mathcal M}
\newcommand{\Span}{\operatorname{span}}

\title{Intersecting families and nonvanishing multivariate polynomials over finite fields}
\author{Shamil Asgarli}
\address{Department of Mathematics \& Computer Science \\ Santa Clara University \\ Santa Clara, CA 95053 \\ United States}
\email{sasgarli@scu.edu}

\author{Bence Csajb\'{o}k}
\thanks{The work of Bence Csajb\'{o}k was supported by the J\'{a}nos Bolyai Research Scholarship of the Hungarian Academy of Sciences and partially by the National Research, Development and Innovation Fund -- grant numbers ADVANCED 153080, EXCELLENCE 151504 and SNN 152582.}
\address{Department of Computer Science \\ E\"otv\"os Lor\'and University \\ H-1117 Budapest, P\'azm\'any P.\ stny.\ 1/C\\ Hungary}
\email{bence.csajbok@ttk.elte.hu}

\author{Chi Hoi Yip}
\address{Department of Mathematics, Hong Kong University of Science and Technology, Clear Water Bay, Hong Kong}
\email{machyip@ust.hk}
\begin{document}

\begin{abstract}
Let $\mathcal{P}_{n,d}$ be the space of polynomials in $n$ variables over
$\mathbb{F}_q$ of degree at most $d$. Two polynomials
$f,g\in\mathcal{P}_{n,d}$ \emph{intersect} if
$f(\mathbf a)=g(\mathbf a)$ for some $\mathbf a\in\mathbb{F}_q^n$. A \emph{star} consists of all polynomials $f\in\mathcal{P}_{n,d}$ satisfying $f(\mathbf a)=b$ for fixed $\mathbf a\in\mathbb{F}_q^n$ and $b\in\mathbb{F}_q$. We completely classify the maximum intersecting families in $\mathcal{P}_{n,d}$. When $n=1$ and $d\geq 2$, it was previously shown that all maximum intersecting families are stars. We prove that the same conclusion holds for all $n\geq 2$ and $d\geq 2$ when $q$ is odd. When $q$ is even, however, the situation is more subtle, and a new phenomenon emerges: for $q\geq 4$, maximum non-star examples exist precisely when $d\leq n$. Along the way, we prove two further results of independent interest. First, we determine the span of nonvanishing polynomials in $\mathcal{P}_{n,d}$. Second, we characterize all linear functionals $\Psi\colon\mathcal{P}_{n,d}\to\mathbb{F}_q$ whose kernels are disjoint from the set of nonvanishing polynomials. The first result plays a crucial role in the proof of our main result; the second is a Gleason--Kahane--\.{Z}elazko theorem for polynomials of bounded degree over finite fields.
\end{abstract}

\maketitle


\section{Introduction}

Let $q$ be a prime power and $n,d\geq 1$. Define
\[
\cP_{n,d}=\F_q[x_1,\ldots,x_n]_{\leq d} \quad \text{and} \quad \cH_{n,d}=\F_q[x_1,\ldots,x_n]_{=d}.
\]
These are the spaces of polynomials of degree at most $d$ and of homogeneous degree-$d$ polynomials in $n$ variables, respectively. Recall that $\dim_{\F_q}\cP_{n,d}=\binom{n+d}{n}$.

Two polynomials $f$ and $g$ are said to \emph{intersect} if their graphs intersect, that is, $f(\mathbf{a})=g(\mathbf{a})$ for some $\mathbf{a}\in \F_{q}^{n}$. An \emph{intersecting family} of polynomials is a subset $\mathcal{F}\subseteq \cP_{n,d}$ such that any two elements of $\mathcal{F}$ intersect. More generally, a family is \emph{$r$-wise intersecting} if any $r$ of its (not necessarily distinct) members take a common value at some point of $\F_q^n$. For $\mathbf a\in \F_q^n$ and $b\in \F_q$, the family
\[
\{f\in\cP_{n,d}:f(\mathbf a)=b\}
\]
is called a \emph{star} centered at $(\mathbf{a},b)$. Every star in $\cP_{n,d}$ has size $q^{\binom{n+d}{n}-1}$. Note that $f$ and $f+c$ cannot belong to the same intersecting family for a nonzero constant $c$. As a result, every intersecting family has size at most $q^{\binom{n+d}{n}-1}$. We say that $\cP_{n,d}$ has the \emph{strict EKR property} if every intersecting family of this maximum size is a star. The main focus of our paper is to understand the extent to which the strict EKR property holds for $\cP_{n,d}$. Here EKR stands for Erd\H{o}s--Ko--Rado.

The classical Erd\H{o}s--Ko--Rado theorem classifies the largest intersecting families of uniform subsets \cite{EKR61}. There are many variants of the EKR theorem tailored for different domains; there are versions for permutations, vector spaces, lattices, and so on. In each case, the main EKR-type question asks whether every maximum intersecting family is a star; see the monograph \cite{GM16} by Godsil and Meagher for a comprehensive treatment of this subject. We briefly mention a few references closely related to our problem. For words over an alphabet of size at least three, maximum intersecting families are stars \cite{Livingston79,M82}; over the binary alphabet, maximum intersecting families need not be stars, while maximum 3-wise families are stars \cite{AY26}. In Hamming space these are anticode problems \cite{AK98}. Evaluating $f\in\cP_{n,d}$ on $\F_q^n$ gives a word, and two polynomial graphs intersect exactly when their words agree in a coordinate. The image is a generalized Reed--Muller code \cite{DGM70}. 

For graphs of one-variable polynomials, Adriaensen \cite{A22} proved the strict EKR theorem in $\cP_{1,d}$ for $d\geq 2$ and all $q$, and stability theorems were obtained by Adriaensen \cite{A22b} and by Aguglia, Csajb\'{o}k, and Weiner \cite{ACW24}. A recent preprint of Adriaensen \cite{Adriaensen25Codes} gives a further stability result and extends the setting to homogeneous bivariate polynomials and linear codes. Related Helly-type questions for bounded-degree hypersurfaces were studied by Deza and Frankl \cite{DF87}, while other polynomial EKR problems use common factors rather than common points of graphs \cite{GMM24,ST25}.

Our main theorem gives the complete strict EKR classification for
every number of variables and every degree.

\begin{thm}\label{thm:main}
Let $q$ be a prime power and let $n,d\geq1$.
\begin{enumerate}[(1)]
\item Every maximum intersecting family in $\cP_{n,d}$ is a star if
and only if one of the following holds:
\[
  q=2\text{ and }n=1;\qquad
  q\text{ is odd and }d\geq2;\qquad
  q\geq4\text{ is even and }d>n.
\]
\item Every maximum 3-wise intersecting family in $\cP_{n,d}$ is
a star, except when $q=2$, $d=1$, and $n\geq2$. In this exceptional
case, every maximum 4-wise intersecting family is a star.
\end{enumerate}
\end{thm}

We also classify maximum intersecting families in every exceptional
range. In degree one, they are parametrized by arbitrary functions on
the coefficient space; over $\F_2$, they arise by choosing one evaluation class from each complementary pair. In even characteristic with $2\leq d\leq n$, new intersecting families arise from the fact that two polynomials are always intersecting if their coefficients differ on some squarefree monomial of degree $d$. This result is achieved by a lemma in the spirit of the Chevalley--Warning theorem (see Lemma~\ref{lem:parity}).

For $n=1$, to establish the strict EKR theorem, previous works~\cite{A22, A22b, ACW24} first established the case $d=2$ using various tools from finite geometry, algebraic graph theory, and number theory. They then completed the proof via an induction on $d$ based on a simple double counting argument. However, this approach does not extend directly to $n\geq 2$. We instead use norm forms and spanning results for nonvanishing polynomials (Section~\ref{sec:span}), together with inductions on both $d$ and $n$ (Section~\ref{subsec:lifting}); the base cases depend on the parity of $q$.

A polynomial $f\in\cP_{n,d}$ is \emph{nonvanishing} if $f(\mathbf{a})\neq 0$ for all $\mathbf{a}\in \F_{q}^{n}$. A family $\mathcal{F}$ is intersecting precisely when no difference $f-g$ (with $f,g\in\mathcal F$) is nonvanishing. To prove Theorem~\ref{thm:main}, the key insight is that there are enough nonvanishing polynomials so that the structure of $\mathcal{F}$ is constrained. It is useful to consider the following subspace $W_{n,d}$ of polynomials:
\[
W_{n,d}=\operatorname{span}_{\F_q}\{f \in\cP_{n,d} : f \text{ is nonvanishing}\}.
\]
The next theorem determines $W_{n,d}$. We use the shorthand $x_{I}=\prod_{i\in I}x_i$ for any subset $I\subseteq [n]=\{1,2,\dots, n\}$, and we write $[x_I]f$ for the coefficient of the monomial $x_{I}$ in $f$.

\begin{thm}\label{thm:span}
Let $n\geq1$ and $d\geq2$. Then
\begin{enumerate}[(1)]
\item If $q=2$, then
\[
W_{n,d} =\{f\in\cP_{n,d}:f\text{ represents a constant function on }\F_2^n\}.
\]
\item If $q\geq4$ is even and $2\leq d\leq n$, then 
\[
W_{n,d}=\{f\in\cP_{n,d}: [x_I]f = 0 \text{ for all subsets } I \text{ of size $d$}\}.
\]
\item If $q$ is odd, or if $q\geq4$ is even and $d>n$, then 
\[
W_{n,d}=\cP_{n,d}.
\]
\end{enumerate}
\end{thm}

We remark that Theorem~\ref{thm:span} alone is not enough to prove Theorem~\ref{thm:main}. It turns out that we need a stronger version (see Section~\ref{sec:span}) that is slightly more technical. Nevertheless, for $d\geq 2$ and $q\geq 3$, the strict EKR property holds precisely when $W_{n,d}=\cP_{n,d}$.

We also classify the linear functionals on $\cP_{n,d}$ whose kernels are disjoint from the set of nonvanishing polynomials in $\cP_{n,d}$. More precisely, we call an $\F_q$-linear functional $\Psi\colon \cP_{n,d}\to \F_q$ \emph{zero-detecting} if
\[
\Psi(f)=0 \quad\Longrightarrow\quad f\text{ has a zero in }\F_q^n.
\]    
Since $1\in \cP_{n,d}$ has no zero, $\Psi(1)\neq 0$. After normalizing, we may assume that $\Psi(1)=1$. 
Note that $\Psi$ is zero-detecting with $\Psi(1)=1$ if and only if for each $f\in \cP_{n,d}$ there exists $\mathbf{a}\in \F_q^n$ such that $\Psi(f)=f(\mathbf{a})$. Indeed, if $\Psi(f)=b$, then $\Psi(f-b)=\Psi(f)-b\Psi(1)=0$, so the zero-detecting property gives $f(\mathbf a)-b=0$ for some $\mathbf a\in\F_q^n$; the converse follows by taking $b=0$.

Observe that the evaluation map $\operatorname{ev}_{\mathbf{a}}\colon \cP_{n,d}\to \F_q$ given by $\operatorname{ev}_{\mathbf{a}}(f)=f(\mathbf{a})$ is an example of a zero-detecting functional. As a consequence of Theorem~\ref{thm:main}, the next result shows that these evaluation maps turn out to be the only (normalized) examples in the case where $q$ is odd, or where $q\geq 4$ is even and $d>n$; however, there are additional functionals when $q\geq 4$ is even and $d\leq n$.

\begin{thm}\label{thm:zero-detecting}
Let $n\geq1$ and $d\geq2$, and let
$\Psi\colon \cP_{n,d}\to \F_q$ be $\F_q$-linear with $\Psi(1)=1$.
\begin{enumerate}[(1)]
\item Let $q\geq 3$. If $q$ is odd or $d>n$, then $\Psi$ is zero-detecting if and only if there is $\mathbf a\in \F_q^n$ such that
\[
  \Psi(f)=f(\mathbf a).
\]
\item If $q\geq 4$ is even and $d\leq n$, then $\Psi$ is zero-detecting
if and only if
\[
  \Psi(f)
  =f(\mathbf a)
   +\sum_{I\in\binom{[n]}{d}}\mu_I \cdot [x_I]f
\]
for some $\mathbf a\in \F_q^n$ and
$\mu_I\in \F_q$.

\item If $q=2$, then $\Psi$ is zero-detecting if and only if there is a subset $S\subseteq\F_2^n$ of odd size such that
\[
  \Psi(f)=\sum_{\mathbf a\in S}f(\mathbf a).
\]
\end{enumerate}
\end{thm}

One motivation for the above theorem comes from the celebrated Gleason--Kahane--\.{Z}elazko (GKZ) theorem \cite{Gle67, KZ68} from functional analysis. It states that a normalized linear functional on a unital complex Banach algebra that is nonzero on every invertible element must be multiplicative. GKZ-type theorems have appeared extensively in other settings. For instance, Mashreghi, Ransford and Ransford \cite{MRR18} proved a Dirichlet space analogue of GKZ: a linear functional on a Dirichlet space that is nonzero on every nowhere-vanishing function must be a multiple of a point evaluation. Mashreghi and Ransford~\cite{MR15} established a similar result for a Hardy space where the condition is on outer functions (instead of nowhere-vanishing functions). Similarly, our Theorem~\ref{thm:zero-detecting} can be viewed as a finite field analogue of the GKZ theorem. It turns out that Theorem~\ref{thm:zero-detecting} is equivalent to classifying $\F_q$-linear maximum intersecting families; indeed, the kernel of a zero-detecting functional is precisely an $\F_q$-linear maximum intersecting family in $\cP_{n,d}$.

\medskip

We interpret our theorems within a graph-theoretic framework. Let $S$ denote the set of all nonvanishing polynomials in $\cP_{n,d}$; evidently, $S$ is closed under multiplication by nonzero scalars. Consider the Cayley graph $G=\operatorname{Cay}(\cP_{n,d}, S)$ with the vertex set $\cP_{n,d}$ such that two polynomials $f$ and $g$ are adjacent exactly when $f-g\in S$; this is equivalent to $f$ and $g$ not intersecting. An intersecting family in $\cP_{n,d}$ corresponds to an independent set in $G$. Theorem~\ref{thm:main} determines precisely when all maximum
independent sets in $G$ correspond to stars. Theorem~\ref{thm:span} determines exactly when $G$ is connected (because connectedness is equivalent to $W_{n,d}=\cP_{n,d}$). When $G$ is not connected, the strict EKR property does not hold, which can be seen by putting together maximum independent sets from each connected component (unless $q=2$ and $n=1$). Theorem~\ref{thm:zero-detecting} characterizes all linear proper colorings $\cP_{n,d}\to \F_q$ of the graph $G$. Since $G$ contains a clique of size $q$ consisting of constant polynomials, the chromatic number of $G$ is exactly $q$. 

\medskip 

The paper is organized into six sections. Section~\ref{sec:preliminaries}
contains the elementary tools used throughout. Section~\ref{sec:remaining-cases} establishes Theorem~\ref{thm:main} (2). Section~\ref{sec:span} studies norm forms, nonvanishing polynomials, and their spanning properties, ultimately proving Theorem~\ref{thm:span}. Section~\ref{sec:qge3-pairwise} determines maximum intersecting families when $q\geq 3$ and $d\geq 2$, completing the proof of Theorem~\ref{thm:main}. Finally, we prove Theorem~\ref{thm:zero-detecting} in Section~\ref{sec:GKZ}.

\medskip

\textbf{Notation.}  Throughout the paper, we adopt the following standard notation. We work with polynomials in $\F_q[x_1,\ldots, x_n]$. We write $\mathbf{x}=(x_1,\ldots, x_n)$. For a multi-index $\boldsymbol{\alpha}=(\alpha_1, \ldots, \alpha_n)$, we write $\mathbf{x}^{\boldsymbol{\alpha}}=x_1^{\alpha_1} \cdots x_n^{\alpha_n}$ and $|\boldsymbol{\alpha}|=\alpha_1+\cdots+\alpha_n$. Given a polynomial $f\in \F_q[x_1,\ldots, x_n]$, we write $[\mathbf{x}^{\boldsymbol{\alpha}}]f$ for the coefficient of the monomial $\mathbf{x}^{\boldsymbol{\alpha}}$ in $f$. We represent squarefree monomials as $x_I=\prod_{i\in I} x_i$ where $I\subseteq [n]=\{1,2,\ldots, n\}$.

Given a field extension $\F_{q^m}/\F_q$, the norm $\Nm_{\F_{q^m}/\F_q}\colon \F_{q^m}\to \F_q$ and the trace $\operatorname{Tr}_{\F_{q^m}/\F_q}\colon \F_{q^m}\to \F_q$ are defined respectively by
\begin{equation*}
\Nm_{\F_{q^m}/\F_q}(x)=\prod_{j=0}^{m-1} x^{q^j},
\qquad \operatorname{Tr}_{\F_{q^m}/\F_q} (x)=\sum_{j=0}^{m-1} x^{q^j}.
\end{equation*}

\section{Preliminaries}\label{sec:preliminaries}

In this section, we collect several preliminary lemmas used throughout the paper.

\begin{lem}\label{lem:normalization}
Let $r\geq2$, let $\mathbf a\in \F_q^n$, and let $h\in\cP_{n,d}$. The map $T\colon\cP_{n,d}\to\cP_{n,d}$ defined by
\[
f(\mathbf x) \mapsto f(\mathbf x+\mathbf a)-h(\mathbf x)
\]
is a bijection that preserves $r$-wise intersection and maps stars
to stars.
\end{lem}

\begin{proof}
For any $f_1,\ldots,f_r$ and $\mathbf x\in\F_q^n$, the values
$T(f_i)(\mathbf x)$ are equal precisely when the values
$f_i(\mathbf x+\mathbf a)$ are equal. The map is bijective and sends
stars to stars.
\end{proof}

The next lemma describes some basic structure of intersecting families.

\begin{lem}\label{lem:transversal}
Let $V\leq\cP_{n,d}$ be an $\F_q$-linear subspace containing the
constant polynomials, and let $C$ be an affine translate of $V$. Every
intersecting family $\cA\subseteq C$ satisfies $|\cA|\leq |V|/q$. If equality holds, then $\cA$ contains exactly one polynomial from
each set
\[
\{f+c:c\in\F_q\},\qquad f\in C.
\]
In particular, for every $\mathbf a\in\F_q^n$ and $b\in\F_q$, the
family $\{f\in C:f(\mathbf a)=b\}$
is a maximum intersecting family in $C$. Moreover, let $U\leq V$ be a subspace containing the constant polynomials. If
$|\cA|=|V|/q$, then every coset $D$ of $U$ in $C$ satisfies $|\cA\cap D|=|U|/q$.
\end{lem}

\begin{proof}
Partition $C$ into the $|V|/q$ cosets of the one-dimensional subspace
of constant polynomials. Distinct members of a coset differ by a
nonzero constant, so their graphs are disjoint. This proves the bound
and the first equality statement.

Fix $\mathbf a\in\F_q^n$ and $b\in\F_q$. Each coset modulo the
constants contains a unique polynomial $f$ satisfying
$f(\mathbf a)=b$. The resulting family is intersecting and attains
the bound.

For any coset $D$ of $U$ in $C$, the same argument gives $|\cA\cap D|\leq |U|/q$. There are $|V|/|U|$ such cosets, and these
upper bounds sum to $|V|/q=|\cA|$, so equality holds for every $D$.
\end{proof}

The following corollary provides a convenient normalization for intersecting families.

\begin{cor}\label{cor:standard-normalization}
Let $V\leq\cP_{n,d}$ be an $\F_q$-linear subspace containing the
constant polynomials and the coordinate polynomials $x_1,\ldots,x_n$, and assume
that $V$ is invariant under translations of the variables. Let
$\cF\subseteq V$ be an $r$-wise intersecting family of size $|V|/q$,
where $r\geq2$. After a normalization as in
Lemma~\ref{lem:normalization}, we may assume that
\[
  0,x_1,\ldots,x_n\in\cF.
\]
Moreover, if $V_0=\{f\in V:f(0)=0\}$, then there is a unique function
$\varphi\colon V_0\to\F_q$ such that
\[
\cF=\{f+\varphi(f):f\in V_0\}, \qquad \varphi(0)=\varphi(x_1)=\cdots=\varphi(x_n)=0.
\]
\end{cor}

\begin{proof}
Subtract a fixed member of $\cF$, so that $0\in\cF$. For each $i$,
Lemma~\ref{lem:transversal} gives a unique $a_i\in\F_q$ such that
$x_i-a_i\in\cF$. Translating the variables by
$\mathbf a=(a_1,\ldots,a_n)$ gives
$0,x_1,\ldots,x_n\in\cF$. Translation invariance ensures that the
transformed family remains in $V$.

By the equality case of Lemma~\ref{lem:transversal}, $\cF$ contains exactly one member of each coset modulo the constants. Every such coset has a unique representative $f\in V_0$, so its member of $\cF$ can be written uniquely as $f+\varphi(f)$. The inclusions $0,x_1,\ldots,x_n\in\cF$ give the stated values of $\varphi$.
\end{proof}

We can precisely compute the difference set of two stars.

\begin{lem}\label{lem:difference-of-stars}
Let $m,D\geq1$, let $\mathbf a,\mathbf b\in \F_q^m$, and let
$c_0,c_1\in \F_q$. Then
\[
 \{Q-P:P,Q\in\cP_{m,D},\ P(\mathbf a)=c_0,\ Q(\mathbf b)=c_1\}
 =
 \begin{cases}
   \cP_{m,D},&\mathbf a\neq\mathbf b,\\[1mm]
   \{R\in\cP_{m,D}:R(\mathbf a)=c_1-c_0\},
      &\mathbf a=\mathbf b.
 \end{cases}
\]
\end{lem}

\begin{proof}
The second case is immediate. Suppose that $\mathbf a\neq\mathbf b$, and let $R\in\cP_{m,D}$. Choose an affine-linear polynomial $P$ satisfying $P(\mathbf{a})=c_0$ and  $P(\mathbf{b})=c_1-R(\mathbf b)$. Setting $Q=P+R$ yields $Q(\mathbf b)=c_1$ and $Q-P=R$.
\end{proof}

The following lemma is a special case of the Combinatorial Nullstellensatz; see \cite{Alon99}.

\begin{lem}\label{lem:grid}
Let $\mathbb{F}$ be a finite field, and let $P\in \mathbb{F}[t_1,\ldots,t_m]$. If
$\deg_{t_i}P<|\mathbb{F}|$ for every $i$ and $P$ vanishes on $\mathbb{F}^m$, then $P=0$.
\end{lem}

The lemma below will be applied frequently in connection with spanning sets of polynomials.

\begin{lem}\label{lem:spanning-invariance}
Let $V$ be a vector space over $\F_q$, and let $\cD\subseteq V$ span $V$ and be closed under multiplication by nonzero elements of $\F_q$. Let $X$ be any set and let $\vartheta\colon V\to X$. If $\vartheta(v+s)=\vartheta(v)$ for all $v\in V$ and $s\in\cD$, then $\vartheta$ is constant.
\end{lem}

\begin{proof}
Given $u,v\in V$, write $u-v=s_1+\cdots+s_m$ with $s_j\in\cD$. This is possible because $\cD$ spans $V$ and is closed under multiplication by nonzero scalars. Applying the hypothesis successively gives $\vartheta(u)=\vartheta(v)$.
\end{proof}

We present a lemma that plays a key role in the even characteristic case. It can be viewed as a boundary case of the Chevalley--Warning theorem. It is also related to the Hasse--Witt point-count congruence in characteristic $2$; see, for example, Katz \cite[Section 2.3.7]{Kat72} and Adolphson--Sperber \cite{AS17}. Although it follows from \cite[Theorem 1.4]{AS17}, we include a self-contained proof below.

\begin{lem}\label{lem:parity}
Let $q=2^r$, write $\mathbf x=(x_1,\ldots,x_d)$, and let $f\in \F_q[x_1,\ldots,x_d]$ with $\deg f\leq d$. Put $a=[x_1x_2\cdots x_d]f$ and let
\[
  Z(f)=\bigl|\{\mathbf x\in \F_q^d:f(\mathbf x)=0\}\bigr|.
\]
Then $Z(f)\equiv a^{q-1}\pmod2$. In particular, if $a\neq0$, then $f$ has a zero in $\F_q^d$.
\end{lem}

\begin{proof}
Let $\overline{Z}(f)$ be the reduction of $Z(f)$ mod $2$. We view $\overline{Z}(f)$ as an element of $\F_2\subseteq\F_q$. Observe that
\[
\overline{Z}(f)=\sum_{\mathbf x\in\mathbb F_q^d}\bigl(1-f(\mathbf x)^{q-1}\bigr).
\]
Moreover,
\[
\sum_{t\in\mathbb F_q}t^m=
\begin{cases}
1,& m>0\ \text{and}\ q-1\mid m,\\
0,&\text{otherwise}.
\end{cases}
\]
Consequently, $\sum_{\mathbf{x}\in\F_q^{d}} \mathbf{x}^{\boldsymbol{\alpha}}=0$ if $\alpha_i<q-1$ for some $i$. As $\deg f^{q-1}\le d(q-1)$, the only monomial of $f^{q-1}$ that
survives after summing over $\mathbb F_q^d$ is
$x_1^{q-1}\cdots x_d^{q-1}$. We obtain
\begin{equation}\label{eq:Z(f)}
\overline{Z}(f)=[x_1^{q-1}\cdots x_d^{q-1}]f^{q-1}.
\end{equation}
It remains to compute this coefficient. Write $f=\sum_{\boldsymbol{\beta}} c_{\boldsymbol{\beta}}\mathbf{x}^{\boldsymbol{\beta}}$ so that $c_{\mathbf 1}=a$, where $\mathbf{1}=(1,\ldots,1)$. A term in the multinomial expansion that
contributes to $\mathbf x^{(q-1)\mathbf1}$ is specified by
nonnegative integers $m_{\boldsymbol\beta}$ satisfying
\[
\sum_{\boldsymbol{\beta}} m_{\boldsymbol{\beta}}=q-1,
\qquad
\sum_{\boldsymbol{\beta}} m_{\boldsymbol{\beta}}\boldsymbol{\beta}=(q-1)\mathbf 1.
\]
Taking total degrees shows that every $\boldsymbol{\beta}$ with $m_{\boldsymbol{\beta}}>0$ satisfies $|\boldsymbol{\beta}|=d$.

Suppose the corresponding multinomial coefficient $\binom{q-1}{(m_{\boldsymbol{\beta}})_{\boldsymbol{\beta}}}$ is odd. By Kummer's theorem, the 
addition $\sum_{\boldsymbol{\beta}}m_{\boldsymbol{\beta}}=q-1$ is carry-free in base $2$. Since \(q-1\) is odd,
exactly one \(m_{\boldsymbol{\beta}}\) is odd. Reducing the second equality modulo \(2\)
shows that its exponent vector satisfies
$\boldsymbol{\beta}\equiv\mathbf 1\pmod 2$. 
But \(|\boldsymbol{\beta}|=d\), so its \(d\) coordinates are positive odd integers with
sum \(d\). Hence \(\boldsymbol{\beta}=\mathbf 1\).

Replace $m_{\mathbf 1}$ by $m_{\mathbf 1}-1$, and then divide all the $m_{\boldsymbol{\beta}}$ values by $2$. The same argument then applies with $q-1$ replaced by $(q-2)/2=2^{r-1}-1$. Repeating this for all binary digits of $q-1$ shows that
\[
m_{\mathbf 1}=q-1, \qquad \text{and}
\qquad
m_{\boldsymbol{\beta}}=0 \quad \text{for all} \quad \boldsymbol{\beta}\ne\mathbf 1.
\]
Therefore, every other multinomial coefficient is even, and
\begin{equation}\label{eq:coeff-c1}
[x_1^{q-1}\cdots x_d^{q-1}]f^{q-1}
=c_{\mathbf 1}^{q-1}=a^{q-1}.
\end{equation}
Combining equations~\eqref{eq:Z(f)} and \eqref{eq:coeff-c1} yields $\overline{Z}(f)=a^{q-1}$, 
as required.
\end{proof}

In Section~\ref{sec:qge3-pairwise}, we will find it helpful to prescribe simultaneously the value of the norm and the value of an $\F_q$-linear functional at a single element of $\F_{q^m}$.

\begin{lem}\label{lem:prescribed-norm-linear}
Let $q\geq4$ and $m\geq3$, and let
$T\colon\F_{q^m}\to\F_q$ be a nonzero $\F_q$-linear map. For every
$a,\tau\in\F_q^\times$, there exists $\gamma\in\F_{q^m}$ such that
\[
  \Nm_{\F_{q^m}/\F_q}(\gamma)=a,
  \qquad
  T(\gamma)=\tau.
\]
\end{lem}

\begin{proof}
By nondegeneracy of the trace pairing, there is
$b\in\F_{q^m}^\times$ such that $T(\gamma)=\Tr_{\F_{q^m}/\F_q}(b\gamma)$. It is therefore enough to find $y\in\F_{q^m}$ satisfying
\[
  \Nm_{\F_{q^m}/\F_q}(y)
  =a\Nm_{\F_{q^m}/\F_q}(b), \qquad  \Tr_{\F_{q^m}/\F_q}(y)=\tau,
\]
and then take $\gamma=b^{-1}y$.

For $u,w\in\F_q^\times$, let
\[
  \nu_m(u,w)
  =\bigl|\{y\in\F_{q^m}:
      \Tr_{\F_{q^m}/\F_q}(y)=u,\ 
      \Nm_{\F_{q^m}/\F_q}(y)=w\}\bigr|.
\]
By the trace--norm estimate of Moisio and Wan
\cite[Theorem~1.2]{MW10},
\[
  \left|
    \nu_m(u,w)-\frac{q^{m-1}-1}{q-1}
  \right|
  \leq (m-1)q^{(m-2)/2}.
\]
We show the main term is larger than the error term. If $m=3$, then
\[
  \frac{q^2-1}{q-1}=q+1>2\sqrt q.
\]
If $m\geq4$, then
\[
  \frac{q^{m-1}-1}{q-1}
  >q^{m-2}
  >(m-1)q^{(m-2)/2},
\]
because $q^{(m-2)/2}\geq2^{m-2}>m-1$. Hence $\nu_m(u,w)>0$, and the required element $y$ exists.
\end{proof}

\section{3-wise EKR and special cases}
\label{sec:remaining-cases}

In this section, we settle the special cases of our main theorem and prove a 3-wise Erd\H{o}s--Ko--Rado theorem for polynomials over finite fields.

\subsection{Degree one}
\label{subsec:degree-one}

\begin{lem}\label{lem:degree-one}
The maximum intersecting families in $\cP_{n,1}$ are precisely the
families
\[
  \cF_\varphi
  =\{\mathbf a\cdot\mathbf x+\varphi(\mathbf a):
       \mathbf a\in\F_q^n\},
\]
where $\varphi\colon\F_q^n\to\F_q$ is arbitrary. Moreover,
$\cF_\varphi$ is a star if and only if there are
$\mathbf s\in\F_q^n$ and $b\in\F_q$ such that $\varphi(\mathbf a)=b-\mathbf a\cdot\mathbf s$ for all $\mathbf a\in\F_q^n$.
\end{lem}

\begin{proof}
By Lemma~\ref{lem:transversal}, a maximum intersecting family contains
exactly one polynomial for each linear part. This allows us to write 
the family as $\cF_\varphi$ for some function $\varphi\colon\F_q^n\to\F_q$. Conversely, every $\cF_\varphi$ is intersecting: the difference of two distinct members is a nonconstant
affine-linear polynomial, which has a zero. As $\cF_\varphi$ has $q^n$ members, it is a maximum family.

The family is the star $\{f \in \mathcal{P}_{n,1}: f(\mathbf s)=b\}$ precisely when $\mathbf a\cdot\mathbf s+\varphi(\mathbf a)=b$ for all $\mathbf a\in\F_q^n$, or equivalently, $\varphi(\mathbf a)=b-\mathbf a\cdot\mathbf s$ for every $\mathbf a$. 
\end{proof}

\begin{prop}\label{prop:degree-one}
The space $\cP_{n,1}$ has the strict EKR property if and only if
$q=2$ and $n=1$.
\end{prop}

\begin{proof}
If $q=2$ and $n=1$, every function $\F_2\to\F_2$ is affine, so
Lemma~\ref{lem:degree-one} shows that every maximum intersecting
family is a star. In every other case, the lemma gives a maximum
family that is not a star: take $\varphi(a)=a^2$ when $n=1$ and
$q\geq3$, and $\varphi(\mathbf a)=a_1a_2$ for
$\mathbf a=(a_1,\ldots,a_n)$ when $n\geq2$.
\end{proof}

\subsection{Intersecting families over \texorpdfstring{$\F_2$}{F2}}
\label{subsec:F2-pairwise}

For each function $\epsilon\colon\F_2^n\to\F_2$ represented by
a polynomial in $\cP_{n,d}$, let
\[
\Lambda_{\epsilon}=\{f\in\cP_{n,d}:f(\mathbf x)=\epsilon(\mathbf x)
\text{ for all }\mathbf x\in\F_2^n\}.
\]
Write $1+\epsilon$ for the complementary function.

\begin{prop}
\label{prop:F2-pairwise}
The maximum intersecting families in $\cP_{n,d}$ over $\F_2$ are
precisely the unions of one set from each complementary pair
$\{\Lambda_{\epsilon},
\Lambda_{1+\epsilon}\}$.
\end{prop}

\begin{proof}
The sets $\Lambda_{\epsilon}$ partition $\cP_{n,d}$, and
$f\mapsto f+1$ pairs them into sets of equal size. No member of one
set intersects a member of its partner, so an intersecting family can meet at most one set from each pair. Choosing one set from each pair accounts for exactly half of $\cP_{n,d}$; a maximum family must therefore contain one entire set from every pair. Conversely, if two polynomials in such a union failed to intersect, their represented functions would be complementary, placing them in the two sets of the same pair, a contradiction.
\end{proof}

\begin{thm}
\label{thm:F2-pairwise}
The space $\cP_{n,d}$ over $\F_2$ has the strict EKR
property if and only if $n=1$.
\end{thm}

\begin{proof}
There are $2^{n+1}$ distinct stars. If $n=1$, the $4$ functions
$\F_2\to\F_2$ are affine-linear and form two complementary pairs.
Proposition~\ref{prop:F2-pairwise} gives $4$ maximum intersecting families, coinciding with the $4$ stars.

Suppose that $n\geq2$. The pairs $\{0,1\}$,
$\{x_i,1+x_i\}$ for $1\leq i\leq n$, and
$\{x_1+x_2,1+x_1+x_2\}$ give $n+2$ distinct complementary pairs.
Proposition~\ref{prop:F2-pairwise} gives at least $2^{n+2}$ maximum
intersecting families, whereas there are only $2^{n+1}$ stars.
\end{proof}

\subsection{\texorpdfstring{$3$}{3}-wise EKR for
  \texorpdfstring{$q\geq 3$}{q >= 3}}
\label{subsec:threewise-qge3}

The maximum size of a 3-wise intersecting family is again $q^{\binom{n+d}{n}-1}$: every such family is intersecting, and every star is 3-wise intersecting.

\begin{prop}\label{prop:threewise-linear}
Let $q\geq3$, and let $\cF\subseteq\cP_{n,d}$ be a maximum
3-wise intersecting family with $0\in\cF$. Then $\cF$ is an
$\F_q$-linear subspace.
\end{prop}

\begin{proof}
We first show that $\mathcal{F}$ is closed under scalar multiplication. Fix $f\in\cF$ and $\lambda\in\F_q$. For every $g\in\cF$, the polynomials $0,f,g$ agree at some input $\mathbf{a}\in\F_{q}^{n}$, and their common value is zero, so $f(\mathbf{a})=g(\mathbf{a})=0$; then $\lambda f$ and $g$ also agree at $\mathbf a$. Hence $\cF\cup\{\lambda f\}$ is intersecting; since $\cF$ has maximum size, we must have $\lambda f\in\cF$.

We next prove closure under differences. Suppose that $g-h\notin\cF$ for some $g,h\in\cF$. The family $\cF\cup\{g-h\}$ is not intersecting because $\cF$ already has
maximum size. Since $\cF$ itself is intersecting, some $f\in\cF$
does not intersect $g-h$; equivalently, $g-h-f$ is nonvanishing.

Since $q>2$, we may choose $\alpha,\beta\in \F_q^{\times}$ with $\alpha+\beta\neq0$. Let $\gamma=-(\alpha+\beta)$, so that
$\alpha,\beta,\gamma$ are nonzero and $\alpha+\beta+\gamma=0$. By closure under scalar multiplication, the three polynomials
\[
  \alpha^{-1}g,\qquad -\beta^{-1}h,\qquad -\gamma^{-1}f
\]
belong to $\cF$. Being three members of a 3-wise intersecting family, they agree at some
$\mathbf u\in\F_q^n$, say $\alpha^{-1}g(\mathbf u)=-\beta^{-1}h(\mathbf u)
=-\gamma^{-1}f(\mathbf u)=c$. Then
$(g-h-f)(\mathbf u)=(\alpha+\beta+\gamma)c=0$, contradicting the
nonvanishing of $g-h-f$.
\end{proof}

\begin{thm}
\label{thm:threewise-qge3}
Let $q\geq3$ and let $n,d\geq1$. Every maximum 3-wise intersecting
family in $\cP_{n,d}$ is a star.
\end{thm}

\begin{proof}
Let $\cF$ be a maximum 3-wise intersecting family. After applying
Corollary~\ref{cor:standard-normalization}, we may assume that
$0,x_1,\ldots,x_n\in\cF$. By Proposition~\ref{prop:threewise-linear}, $\cF$ is an
$\F_q$-linear subspace.

We claim that $\cF$ is closed under multiplication by an arbitrary
polynomial $g$, provided the product has degree at most $d$:
\begin{equation}\label{eq:truncated-ideal}
  f\in\cF,\quad \deg(fg)\leq d
  \quad\Longrightarrow\quad fg\in\cF.
\end{equation}
To prove this, let $f\in\cF$ and let $g$ satisfy $\deg(fg)\leq d$.
Lemma~\ref{lem:transversal} gives a unique $c\in\F_q$ such that
$fg+c\in\cF$. The three polynomials $0,f,fg+c$ agree at some
$\mathbf u\in\F_q^n$. Their common value is zero:
$f(\mathbf u)=0$ and $f(\mathbf u)g(\mathbf u)+c=0$. These equalities
force $c=0$, proving \eqref{eq:truncated-ideal}.

Every nonconstant monomial $\mu$ of degree at most $d$ is divisible
by some $x_i$. Since $x_i\in\cF$, \eqref{eq:truncated-ideal} gives
$\mu\in\cF$. Since $\cF$ is linear, it contains the star $\{f\in\cP_{n,d}:f(0)=0\}$,
which has the same size as $\cF$. The two families are equal, and $\cF$ is a star.
\end{proof}

\subsection{3-wise EKR over \texorpdfstring{$\F_2$}{F2} in degree at least two}
\label{subsec:threewise-F2}

The proof of Proposition~\ref{prop:threewise-linear} uses three nonzero scalars whose sum is zero; no such scalars exist over $\F_2$. We use the following closure properties instead.

\begin{lem}\label{lem:F2-boolean-closure}
Let $\cF\subseteq\cP_{n,d}$ be a maximum 3-wise intersecting family
over $\F_2$.
\begin{enumerate}[(1)]
\item If $0,f,g\in\cF$ and $\deg f+\deg g\leq d$, then
\[
  \Span_{\F_2}\{f,g,fg\}\subseteq\cF.
\]
\item If $f,f+h\in\cF$ and $\deg g+\deg h\leq d$, then
\[
  f+gh\in\cF.
\]
\end{enumerate}
\end{lem}

\begin{proof}
By Lemma~\ref{lem:transversal}, exactly one of $u$ and $u+1$ belongs
to $\cF$ for every $u\in\cP_{n,d}$.

For part~(1), suppose first that $fg\notin\cF$. Then $fg+1\in\cF$, so the three polynomials $0,f,fg+1$ agree at some input $\mathbf{a}\in\F_{2}^{n}$. Their common value is zero, but $f(\mathbf a)=0$
gives $(fg+1)(\mathbf a)=1$, a contradiction. This proves that $fg\in\cF$. If $fg+f\notin\cF$, then $fg+f+1\in\cF$. Now, the three polynomials $0,f,fg+f+1$ agree at some input $\mathbf{b}\in\F_{2}^{n}$; we have $f(\mathbf b)=0$ but $(fg+f+1)(\mathbf b)=1$, another contradiction. We have proved that $fg,fg+f\in\cF$; symmetry gives $fg+g\in\cF$.

If $f+g+fg\notin\cF$, then $f,g,f+g+fg+1$ agree at some input $\mathbf{a}\in\F_{2}^{n}$. Writing $c=f(\mathbf a)=g(\mathbf a)$, the third polynomial takes the value $c+c+c^2+1=c+1$ there, a contradiction. It remains to prove that $f+g\in\cF$. Otherwise, $f+g+1\in\cF$, and $f+fg$, $g+fg$, and $f+g+1$ agree at some input $\mathbf{b}\in\F_{2}^{n}$. Equality of the first two values gives $f(\mathbf b)=g(\mathbf b)$. The first polynomial then has value $f(\mathbf b)+f(\mathbf b)^2=0$, whereas the third has value $1$, a contradiction. This proves part~(1).

For part~(2), suppose that $f+gh\notin\cF$. Then $f+gh+1\in\cF$. Then $f,f+h,f+gh+1$ agree at some input $\mathbf{a}\in\F_{2}^{n}$. Subtracting the value of $f(\mathbf{a})$ gives simultaneously $h(\mathbf{a})=0$ and $g(\mathbf{a})h(\mathbf{a})+1=0$, which is impossible.
\end{proof}

\begin{thm}\label{thm:threewise-F2}
Let $n\geq 1$ and $d\geq2$. Every maximum 3-wise intersecting family
in $\cP_{n,d}$ over $\F_2$ is a star.
\end{thm}

\begin{proof}
Let $\cF$ be a maximum 3-wise intersecting family. After applying
Corollary~\ref{cor:standard-normalization}, we may assume that
$0,x_1,\ldots,x_n\in\cF$. Repeatedly applying
Lemma~\ref{lem:F2-boolean-closure}\textup{(1)} to the coordinate
polynomials shows that every linear form belongs to $\cF$.

We claim that, for $0\leq r\leq n$, every linear form $w$ and all
polynomials $s_1,\ldots,s_r$ of degree at most $d-1$ satisfy
\begin{equation}\label{eq:F2-sum-products}
  w+\sum_{i=1}^r s_ix_i\in\cF.
\end{equation}
We argue by induction on $r$. The case $r=0$ was just proved. Assume
the claim for $r-1$, and put
\[
  p=w+\sum_{i=1}^{r-1}s_ix_i.
\]
The induction hypothesis, applied once with $w$ and once with
$w+x_r$, gives $p,p+x_r\in\cF$. Applying
Lemma~\ref{lem:F2-boolean-closure}\textup{(2)} with $f=p$, $h=x_r$,
and $g=s_r$ gives $p+s_rx_r\in\cF$, proving
\eqref{eq:F2-sum-products}.

Let $f\in\cP_{n,d}$ have zero constant term and linear part $w$. For each monomial term of $f$ of degree at least two, choose a variable $x_i$ dividing it. Collecting the monomials by their chosen variable and factoring $x_i$ out of the $i$th group gives
\[
  f=w+\sum_{i=1}^n s_ix_i,
\]
where $\deg s_i\leq d-1$. Equation~\eqref{eq:F2-sum-products} gives $f\in\cF$. The polynomials with zero constant term form a star of the same size as $\cF$, so the two families are equal.
\end{proof}

\subsection{4-wise EKR for linear polynomials over \texorpdfstring{$\F_2$}{F2}}
\label{subsec:affine-F2}

Theorems~\ref{thm:threewise-qge3} and~\ref{thm:threewise-F2} leave open only the case $d=1$ over $\F_2$, which is exceptional.

\begin{thm}\label{thm:affine-F2}
Let $n\geq1$. The following statements hold for families in $\cP_{n,1}$ over $\F_2$.
\begin{enumerate}[(1)]
\item Pairwise intersection implies 3-wise intersection.
\item Every maximum 3-wise intersecting family is a star if and only
if $n=1$.
\item Every maximum 4-wise intersecting family is a star.
\end{enumerate}
\end{thm}

\begin{proof}

(1) Let $\cF$ be an intersecting family, and choose
$f_1,f_2,f_3\in\cF$. Put $g=f_2-f_1$ and $h=f_3-f_1$. By
Lemma~\ref{lem:normalization}, the family $\cF-f_1$ is intersecting
and contains $0,g,h$. It is enough to show that $g$ and $h$ have a
common zero.

Pairwise intersection gives a zero of each of $g$, $h$, and $g+h$,
so none is the constant polynomial $1$. If the equations $g=0$ and
$h=0$ were inconsistent, then $1$ would lie in their
$\F_2$-linear span, forcing one of $g,h,g+h$ to equal $1$, a
contradiction.

(2) By part~\textup{(1)}, pairwise and 3-wise intersection are equivalent
for families in $\cP_{n,1}$ over $\F_2$. The assertion follows from
Proposition~\ref{prop:degree-one}.

(3) Let $\cF$ be a maximum 4-wise intersecting family. A star is 4-wise
intersecting and has size $2^n$, while Lemma~\ref{lem:transversal} bounds every
intersecting family by $2^n$. Maximality gives $|\cF|=2^n$, and equality in the
lemma gives exactly one member with each linear part:
$\cF=\{\mathbf a\cdot\mathbf x+\varphi(\mathbf a):\mathbf a\in\F_2^n\}$ for some
function $\varphi\colon\F_2^n\to\F_2$. After subtracting $\varphi(0)$ from every
member, we may assume that $\varphi(0)=0$.

For any $\mathbf a,\mathbf b\in\F_2^n$, the $4$ members indexed by $0,\mathbf a,\mathbf b,\mathbf a+\mathbf b$ agree at some $\mathbf u\in\F_2^n$, and the common value is zero because the member indexed by $0$ is the zero polynomial. Adding $\mathbf a\cdot\mathbf u+\varphi(\mathbf a)=0$ and $\mathbf b\cdot\mathbf u+\varphi(\mathbf b)=0$, and comparing with $(\mathbf a+\mathbf b)\cdot\mathbf u+\varphi(\mathbf a+\mathbf b)=0$, gives $\varphi(\mathbf a+\mathbf b)=\varphi(\mathbf a)+\varphi(\mathbf b)$. So $\varphi$ is linear, say $\varphi(\mathbf a)=\mathbf a\cdot\mathbf s$, and the normalized family is $\cF=\{\mathbf a\cdot(\mathbf x+\mathbf s):\mathbf a\in\F_2^n\}$, the star centered at $(\mathbf s,0)$. Undoing the constant translation shows that the original family is also a star.
\end{proof}

Theorem~\ref{thm:main}(2) summarizes the results in this section.

\begin{proof}[Proof of Theorem~\ref{thm:main}(2)]
For 3-wise intersection, Theorem~\ref{thm:threewise-qge3} covers $q\geq3$, and Theorem~\ref{thm:threewise-F2} covers $q=2$ and $d\geq 2$. The remaining case, $q=2$ and $d=1$, is described by Theorem~\ref{thm:affine-F2}, which also gives the 4-wise assertion. 
\end{proof}

\section{Norm forms and nonvanishing polynomials}\label{sec:span}

This section studies spanning properties of norm forms and nonvanishing polynomials in $\cP_{n,d}$. Besides proving Theorem~\ref{thm:span}, we develop additional technical tools used in Section~\ref{sec:qge3-pairwise}.

For an integer $d\geq 2$, write $\Nm=\Nm_{\F_{q^d}/\F_q}\colon \F_{q^d}\to \F_q$ for the norm map. Its restriction $\Nm\colon \F_{q^d}^\times\to\F_q^\times$ is
surjective, and every fiber has size $(q^d-1)/(q-1)=1+q+\cdots+q^{d-1}$. For
$\boldsymbol{\alpha}=(\alpha_1,\ldots,\alpha_n)\in \F_{q^d}^n$, define the homogeneous degree-$d$ form
\begin{equation}\label{eq:norm-form}
  H_{\boldsymbol{\alpha}}(x_1,\ldots,x_n)
  =\Nm(\alpha_1x_1+\cdots+\alpha_nx_n)
  =\prod_{j=0}^{d-1}
     (\alpha_1^{q^j}x_1+\cdots+\alpha_n^{q^j}x_n).
\end{equation}
Observe that  $H_{\boldsymbol{\alpha}}(x_1,\ldots,x_n)\in\cH_{n,d}$. Indeed, the $q$-power Frobenius map permutes the factors of the product cyclically and so fixes each coefficient of $H_{\boldsymbol{\alpha}}$.

\begin{defn}
We call $H_{\boldsymbol{\alpha}}\in \cH_{n,d}$ in \eqref{eq:norm-form} a \emph{proper norm form} if $\Span_{\F_q}\{\alpha_1,\ldots,\alpha_n\}$ is a proper subspace of $\F_{q^d}$.
\end{defn}

First, we prove that every proper norm form can be turned into a nonvanishing polynomial after adding a lower-degree term whose value at a prescribed point can be specified in advance.

\begin{lem} \label{lem:nonvanishing-completion}
Let $H$ be a proper norm form of degree $d$. For every $\mathbf b\in\F_q^n$ and every $\delta\in\F_q$ with $\delta\neq-H(\mathbf b)$, there is $R\in\cP_{n,d-1}$ such that $H+R$ is nonvanishing and $R(\mathbf b)=\delta$.
\end{lem}

\begin{proof}
Let $H(\mathbf{x})=\Nm(\alpha_1x_1+\cdots+\alpha_nx_n)$. Put $U=\Span_{\F_q}\{\alpha_1,\ldots,\alpha_n\}$ and $u= b_1\alpha_1+\cdots+b_n\alpha_n$. Since $U$ is a proper subspace, each nonzero norm fiber has size $\frac{q^{d}-1}{q-1}>q^{d-1}\geq |U|$. Using $\delta+\Nm(u)\neq 0$, we choose $t\notin U$ with
$\Nm(t)=\delta+\Nm(u)$. Define
\[
  R(\mathbf x)
  =\Nm\!\left(t-u+\alpha_1x_1+\cdots+\alpha_nx_n\right)
   -H(\mathbf x).
\]
The leading homogeneous parts cancel, so $R\in\cP_{n,d-1}$. Moreover, $t-u\notin U$, so $H+R$ is nonvanishing, and $R(\mathbf b)=\Nm(t)-\Nm(u)=\delta$.
\end{proof}

The next three results concern spanning properties of norm forms.

\begin{lem}\label{lem:norm-coefficient-separation}
Let $s\geq1$, let $\mathbf y=(y_1,\ldots,y_s)$, and let
$\boldsymbol{\beta}=(\beta_1,\ldots,\beta_s)\in \F_{q^d}^s$. As $\boldsymbol{\beta}$ ranges over $\F_{q^d}^s$, the forms $\Nm(\beta_1y_1+\cdots+\beta_sy_s)$ span $\cH_{s,d}$.
\end{lem}

\begin{proof}
The coefficients of a norm form $\Nm(\beta_1y_1+\cdots+\beta_sy_s)$ can be computed explicitly. For a multi-index $\mathbf e=(e_1,\ldots,e_s)$ with
$|\mathbf e|=d$, the coefficient of
$\mathbf y^{\mathbf e}$ is
\begin{equation}\label{eq:coefficient-polynomial}
 c_{\mathbf e}(\boldsymbol{\beta})=\sum_{\substack{I_1\sqcup\cdots\sqcup I_s=\{0,\ldots,d-1\}\\
                   |I_i|=e_i}}
    \prod_{i=1}^s
    \beta_i^{\sum_{j\in I_i}q^j}.
\end{equation}
We view $c_{\mathbf{e}}$ as a polynomial in variables $\beta_1, \beta_2, \ldots, \beta_s$. The base-$q$ expansion of the exponent
$\sum_{j\in I_i}q^j$ recovers the set $I_i$. If $\mathbf e\neq\mathbf f$, then $c_{\mathbf e}$ and $c_{\mathbf f}$ do not share any common monomial in $\beta_1, \ldots, \beta_{s}$. It follows that the polynomials $c_{\mathbf{e}}$ are $\F_q$-linearly independent. Moreover, the degree of $c_{\mathbf{e}}$ in each variable is at most $\sum_{j=0}^{d-1}q^j=\frac{q^d-1}{q-1}<q^d$. If an $\F_q$-linear combination of them vanishes on $\F_{q^d}^s$, then this combination is the zero polynomial by  Lemma~\ref{lem:grid}. Thus, the functions $c_{\mathbf{e}}\colon\F_{q^d}^s\to\F_q$ with $|\mathbf{e}|=d$ are linearly independent. The matrix $\bigl(c_{\mathbf{e}}(\boldsymbol\beta)\bigr)_{\boldsymbol\beta\in \F_{q^d}^s,\,|\mathbf{e}|=d}$ has full column rank, which coincides with $\dim_{\F_q}\cH_{s,d}$. Its row indexed by $\boldsymbol\beta$ is the coefficient vector of $\Nm(\beta_1y_1+\cdots+\beta_sy_s)$, so these norm forms span $\cH_{s,d}$.
\end{proof}

\begin{prop}
\label{prop:proper-norm-generation}
If $q$ is odd or $d>n$, then the proper norm forms span $\cH_{n,d}$ and are closed under multiplication by elements of $\F_q^\times$.
\end{prop}

\begin{proof}
If $d>n$, then every norm form is proper, since its parameters span a subspace of $\F_{q^d}$ of dimension at most $n<d=[\F_{q^d}:\F_q]$. The spanning assertion follows from Lemma~\ref{lem:norm-coefficient-separation}.

Assume from now on that $q$ is odd and $d\leq n$. Let $W_{\mathrm{prop}}$ be the span of the proper norm forms. Every degree-$d$ monomial supported on fewer than $d$ variables belongs to $W_{\mathrm{prop}}$. Indeed, let $I\subseteq[n]$ be its support. The norm forms in the variables $(x_i)_{i\in I}$ are proper, since $|I|<d$, and by Lemma~\ref{lem:norm-coefficient-separation} they span all degree-$d$ forms in those variables.

It remains to demonstrate that each squarefree degree-$d$ monomial is in $W_{\mathrm{prop}}$. After
relabeling, it suffices to show $x_1\cdots x_d \in W_{\mathrm{prop}}$. For
$t,\alpha_3,\ldots,\alpha_d\in \F_{q^d}$, the norm form
\[
H=\Nm(tx_1+tx_2+\alpha_3x_3+\cdots+\alpha_dx_d)
\]
is proper, because its parameters span a space of dimension at most $d-1$. View $t,\alpha_3,\ldots,\alpha_d$ as variables. Note that $\Gamma(t,\alpha_3,\ldots,\alpha_d):=[x_1\cdots x_d] H$ is a polynomial whose degree in each variable is smaller than $q^d$, and it contains the term $2t^{1+q}\prod_{i=3}^d\alpha_i^{q^{i-1}}$. Since $q$ is odd, this coefficient does not vanish. Thus, $\Gamma$ is a nonzero polynomial and Lemma~\ref{lem:grid} gives a choice of parameters for which $\Gamma\neq0$. Every monomial in $H$ other than $x_1\cdots x_d$ is not squarefree and belongs to $W_{\mathrm{prop}}$ by the previous paragraph. Hence $x_1\cdots x_d\in W_{\mathrm{prop}}$. 

Finally, if $\lambda\in\F_q^\times$ and
$H=\Nm(\alpha_1x_1+\cdots+\alpha_nx_n)$ is proper, choose
$u\in \F_{q^d}^\times$ with $\Nm(u)=\lambda$. Then $\lambda H=\Nm(u\alpha_1x_1+\cdots+u\alpha_nx_n)$ and the norm form on the right is again proper.
\end{proof}

\begin{lem}\label{lem:affine-basis-norm-span}
Let $q\geq 3$. As $(\alpha_0,\ldots,\alpha_n)$ ranges over the ordered
$\F_q$-bases of $\F_{q^{n+1}}$, the nonvanishing polynomials $\Nm_{\F_{q^{n+1}}/\F_q}(\alpha_0+\alpha_1x_1+\cdots+\alpha_nx_n)$ span $\cP_{n,n+1}$ and are closed under multiplication by elements of
$\F_q^\times$.
\end{lem}

\begin{proof}
Each polynomial $\Nm_{\F_{q^{n+1}}/\F_q}(\alpha_0+\alpha_1x_1+\cdots+\alpha_nx_n)$ is nonvanishing, because an $\F_{q}$-zero of such a polynomial implies that $\alpha_0, \ldots, \alpha_n$ are $\F_q$-linearly dependent. To prove spanning, suppose that an $\F_q$-linear functional $L$ on $\cP_{n,n+1}$ annihilates all these polynomials as $(\alpha_0,\ldots,\alpha_n)$ ranges over ordered
$\F_q$-bases of $\F_{q^{n+1}}$. Set
\[
  P(\alpha_0,\ldots,\alpha_n)
  =L\left(\Nm_{\F_{q^{n+1}}/\F_q}
     (\alpha_0+\alpha_1x_1+\cdots+\alpha_nx_n)\right).
\]
Write
\[
D=1+q+\cdots+q^n=\frac{q^{n+1}-1}{q-1}.
\]
View $P$ as a polynomial over $\F_{q^{n+1}}$; it has degree at most $D$ by the same reasoning as in \eqref{eq:coefficient-polynomial}. By hypothesis on $L$, the polynomial $P$ vanishes at every ordered $\F_q$-basis of
$\F_{q^{n+1}}$, and the number of such bases is
\[
\begin{aligned}
  \prod_{i=0}^{n}(q^{n+1}-q^i)
  &=q^{(n+1)^2}\prod_{j=1}^{n+1}(1-q^{-j})\geq
    q^{(n+1)^2}\left(1-\sum_{j=1}^{n+1}q^{-j}\right)\\
  &=q^{(n+1)^2}\frac{q-2+q^{-(n+1)}}{q-1}>q^{(n+1)^2}\frac{1-q^{-(n+1)}}{q-1}
   =Dq^{n(n+1)}.
\end{aligned}
\]
Here we use Bernoulli's inequality, and the strict inequality uses the assumption $q\geq3$. Thus the Schwartz--Zippel lemma \cite{S80} implies that $P$ is the zero polynomial.

It follows that $L$ annihilates every affine norm polynomial
\begin{equation}\label{eq:affine-norm-form}
\Nm_{\F_{q^{n+1}}/\F_q}
(\alpha_0+\alpha_1x_1+\cdots+\alpha_nx_n)
\end{equation}
for every choice of $\alpha_0,\ldots,\alpha_n\in\F_{q^{n+1}}$ (not just for ordered $\F_q$-bases of $\F_{q^{n+1}}$).

Homogenization gives a linear isomorphism
\[
  \cP_{n,n+1}\to\cH_{n+1,n+1},
  \qquad
  f(\mathbf x)\mapsto
  y_0^{n+1}f\!\left(\frac{y_1}{y_0},\ldots,
                     \frac{y_n}{y_0}\right).
\]
Under this isomorphism, the norm polynomial in \eqref{eq:affine-norm-form} corresponds to
\[
  \Nm_{\F_{q^{n+1}}/\F_q}  (\alpha_0y_0+\alpha_1y_1+\cdots+\alpha_ny_n).
\]
These homogeneous norm forms span $\cH_{n+1,n+1}$ by
Lemma~\ref{lem:norm-coefficient-separation}. As their
dehomogenizations span $\cP_{n,n+1}$, we obtain $L=0$. As $L$ is arbitrary, the claim about spanning is complete. 

Finally, if $\lambda\in\F_q^\times$, choose
$t\in\F_{q^{n+1}}^\times$ with
$\Nm_{\F_{q^{n+1}}/\F_q}(t)=\lambda$. Scaling every $\alpha_i$ by
$t$ preserves the basis condition and multiplies the norm polynomial
by $\lambda$.
\end{proof}

Now we are ready to present the proof of Theorem~\ref{thm:span}.

\begin{proof}[Proof of Theorem~\ref{thm:span}]
(1) Suppose $q=2$. Every nonvanishing polynomial takes the value $1$ at every point of $\F_2^n$, so $W_{n,d}$ consists only of polynomials representing constant functions. Conversely, let $g\in\cP_{n,d}$ represent a constant function on $\F_2^n$. If $g$ takes the value $1$, then $g\in W_{n,d}$ by definition. If $g$ takes the value $0$, then both $1$ and $1+g$ are nonvanishing, and therefore $g=(1+g)+1$ belongs to $W_{n,d}$. 

(2) Suppose that $q\geq 4$ is even and $2\leq d\leq n$. Put
\[
K=\{f\in\cP_{n,d}: [x_I]f = 0 \text{ for all subsets } I \text{ of size $d$}\}.
\]
If $f$ is nonvanishing and $[x_I]f\neq0$ for some $I$ of size $d$, then setting $x_j=0$ for all $j\notin I$ leaves a polynomial of degree at most $d$ in the $d$ variables indexed by $I$ with nonzero $x_I$-coefficient, and Lemma~\ref{lem:parity} produces a zero of it, hence a zero of $f$. So every nonvanishing polynomial lies in the subspace $K$, whence $W_{n,d}\subseteq K$. Conversely, let $f\in K$, so that every monomial occurring in $f$ involves at most $d-1$ variables. For each $S\in\binom{[n]}{d-1}$, the polynomials of degree at most $d$ in the variables indexed by $S$ form a copy of $\cP_{d-1,d}$, which is spanned by nonvanishing polynomials by Lemma~\ref{lem:affine-basis-norm-span}; these remain nonvanishing when regarded as polynomials in all $n$ variables. As every monomial of $f$ lies in $W_{n,d}$, we conclude that $K=W_{n,d}$.

(3) Suppose first that $q$ is odd. We start with $d=2$. Nonzero constants are nonvanishing, and so they belong to $W_{n,2}$. Fix a nonsquare
$\eta\in\mathbb F_q$. For each $i$ and $j$, the three polynomials $x_i^2-\eta$, $x_j^2-\eta$, and $(x_i+x_j)^2-\eta$ are nonvanishing; as a result, $x_i^2, x_j^2, (x_i+x_j)^2\in W_{n,2}$. Together with the identity $2x_ix_j = (x_i+x_j)^2-x_i^2-x_j^2$, we get $x_ix_j\in W_{n,2}$ (here we crucially used $2\neq 0$). Considering $(x_i+1)^2-\eta$, we can similarly show that $(x_i+1)^2\in W_{n,2}$, and so $x_i\in W_{n,2}$. We deduce that $W_{n,2}=\mathcal P_{n,2}$.

Now proceed by induction on $d$. In the even-characteristic case
$d>n$, start instead at $d=n+1$, where the assertion is Lemma~\ref{lem:affine-basis-norm-span}.
Thus, in either case, at the inductive step we may assume that
$\mathcal P_{n,d-1}\subseteq W_{n,d-1}\subseteq W_{n,d}$.
Let $H$ be a proper norm form of degree $d$. By Lemma~\ref{lem:nonvanishing-completion}, there is
$R\in\mathcal P_{n,d-1}\subseteq W_{n,d}$ such that $H+R$ is nonvanishing. Thus, $H+R\in W_{n,d}$, and since $R\in W_{n,d}$, we also have $H\in W_{n,d}$. By Proposition~\ref{prop:proper-norm-generation}, the proper norm forms span $\mathcal H_{n,d}$. We have shown that $\mathcal H_{n,d}\subseteq W_{n,d}$. Together with $\mathcal P_{n,d-1}\subseteq W_{n,d}$, this gives
\[
\cP_{n,d}
 =\cP_{n,d-1}+\cH_{n,d}
 \subseteq W_{n,d},
\]
and we conclude that $W_{n,d}=\cP_{n,d}$.
\end{proof}

\begin{rem}
Using Theorem~\ref{thm:span}, one can compute the exact codimension of $W_{n,d}$ in $\cP_{n, d}$. Clearly, $\operatorname{codim}  W_{n,d}=0$ for $q$ odd, or $q\ge 4$ even and $d>n$. Similarly, $\operatorname{codim} W_{n,d}=\binom{n}{d}$ for $q\geq 4$ even and $2\leq d\leq n$. Finally, for $q=2$,
\[
\operatorname{codim} W_{n,d} = \displaystyle \sum_{i=1}^{\min(d,n)} \binom{n}{i}.
\]
Indeed, over $\mathbb F_2$, every polynomial $f\in\cP_{n,d}$ has a unique squarefree reduction $\bar f$, and $f$ is nonvanishing if and only if $\bar f=1$. Hence the quotient space $\cP_{n,d}/W_{n,d}$ has a basis formed by nonconstant squarefree monomials $x_I$ with $1\le |I|\le \min(d,n)$. 
\end{rem}

\section{EKR over fields of order at least three}
\label{sec:qge3-pairwise}

In this section, we complete the proof of Theorem~\ref{thm:main}. 

In the introduction, we mentioned stars in $\cP_{n,d}$. It will be convenient to generalize the definition as follows. If $C$ is an affine translate of a subspace of $\cP_{n,d}$ containing the constant polynomials, we call
\[
  \{f\in C:f(\mathbf a)=b\}
\]
a \emph{star} in $C$, and call $\mathbf a$ its \emph{input center}
(\emph{in-center}, for short).

\subsection{Degree and variable lifting} \label{subsec:lifting}

We develop two ways to lift the star classification to $\cP_{n,d}$:
from $\cP_{n,d-1}$ by increasing the degree and from $\cP_{n-1,d}$ by increasing the number of variables. In both arguments, a maximum intersecting family is decomposed into slices along the cosets of a smaller space $U$. The star classification in $U$ makes each slice a star, while the spanning properties of nonvanishing polynomials force their in-centers to agree.

\begin{lem}
\label{lem:star-slice-synchronization}
Let $U\leq V$ be $\F_q$-linear subspaces of $\cP_{n,d}$. Suppose $U$ contains the constant polynomials, and $U$ can be identified with $\cP_{m,D}$ after choosing $m$ of the $n$ variables (where $m\geq 1$ and $D\geq 1$). Suppose that every maximum intersecting family in $U$ is a star. Assume that $V$ is the $\F_q$-span of $U$ and $\mathcal{Z}$ for some collection $\mathcal Z\subseteq V$ of nonvanishing polynomials. Then for every maximum intersecting family $\cA\subseteq V$, all the intersections of $\cA$ with the $U$-cosets in $V$ are stars with the same in-center.

Moreover, suppose that this common in-center is the origin. Put
$V_0=\{f\in V:f(0)=0\}$ and write
$\cA=\{f+\varphi(f):f\in V_0\}$ for some function $\varphi\colon V_0\to\F_q$.
Then
\[
\varphi(f+u)=\varphi(f),
\qquad \forall f\in V_0,\ u\in U\cap V_0.
\]
\end{lem}
\begin{proof}
By Lemma~\ref{lem:transversal}, $\cA$ meets every $U$-coset in a maximum intersecting family. This family must be a star, because subtracting a coset representative allows us to move the family into $U$ where we can apply the hypothesis. A star in a coset of $U$ determines its in-center uniquely as a point of $\F_q^m$; we can view this point in $\F_q^n$ by setting the remaining coordinates equal to $0$.

For a coset $C$ of $U$, let $\boldsymbol\sigma(C)$ denote the in-center of
$\mathcal{A}\cap C$. Fix $H\in\mathcal Z$ and write $C=g+U$ where $g\in V$.
Suppose, to the contrary, that $\boldsymbol\sigma(C)\neq \boldsymbol\sigma(C+H)$.
Write 
\begin{align*}
\cA\cap C & =g+\{u\in U:u(\boldsymbol\sigma(C))=c_1\},  \\
\cA\cap(C+H)&=g+H+\{u\in U:u(\boldsymbol\sigma(C+H))=c_2\}
\end{align*}
with $c_1,c_2\in\F_q$. As the two in-centers differ,
Lemma~\ref{lem:difference-of-stars}, applied inside $U\cong\cP_{m,D}$, shows
that the difference set of the stars $\{u\in U:u(\boldsymbol\sigma(C))=c_1\}$ and $\{u\in U:u(\boldsymbol\sigma(C+H))=c_2\}$ is all of $U$. In particular, the element $0\in U$ is in the difference set; we deduce that the two stars share an element $u$. Then $P\colonequals g+u$, $Q \colonequals g+H+u$ satisfy $P\in\cA\cap C$, $Q\in\cA\cap(C+H)$, and $Q-P=H$.
As $H$ is nonvanishing, this contradicts the intersecting property of $\cA$. Thus, $\boldsymbol\sigma(C+H)=\boldsymbol\sigma(C)$.

Since $\lambda H$ is nonvanishing for every $\lambda\in\F_q^\times$,
the same argument gives
$\boldsymbol\sigma(C+\lambda H)=\boldsymbol\sigma(C)$. As the images of the elements of $\mathcal Z$ span $V/U$, it follows from Lemma~\ref{lem:spanning-invariance} that $\boldsymbol\sigma$ is a constant.

For the final assertion, let $f\in V_0$ and $u\in U\cap V_0$.
Since $U$ contains the constants, $f+\varphi(f)$ and $(f+u)+\varphi(f+u)$ are in the same $U$-coset; they take the same value at the in-center $0$. Using $u(0)=0$, this leads to $\varphi(f+u)=\varphi(f)$.
\end{proof}

\begin{prop}[Degree lifting]\label{prop:degree-lifting}
Let $q$ be any prime power, let $n\geq1$ and $d\geq2$, and assume that
either $q$ is odd or $d>n$. If every maximum intersecting family in
$\cP_{n,d-1}$ is a star, then every maximum intersecting family in
$\cP_{n,d}$ is a star.
\end{prop}

\begin{proof}
Let $\cF\subseteq\cP_{n,d}$ be a maximum intersecting family, and let
$\mathcal Z\subseteq\cP_{n,d}$ be the following subset:
\[
  \mathcal Z
  =\{H+R:H \text{ is a proper norm form in } \cH_{n,d},\ R\in\cP_{n,d-1},
       \ H+R\text{ is nonvanishing}\}.
\]
By combining Lemma~\ref{lem:nonvanishing-completion} and Proposition~\ref{prop:proper-norm-generation}, we have
\[
  \cP_{n,d}
  =\cP_{n,d-1}+\Span_{\F_q}\mathcal Z.
\]
For $A\in\cH_{n,d}$, write $\cF_A=\{P\in\cP_{n,d-1}:A+P\in\cF\}$. Lemma~\ref{lem:star-slice-synchronization}, applied with $U=\cP_{n,d-1}$, shows that there exists a common in-center $\mathbf s\in\F_q^n$ and a value $b_A\in\F_q$ for each $A$ satisfying
\[
\cF_A=\{P\in\cP_{n,d-1}:P(\mathbf s)=b_A\}.
\]

Fix $A\in\cH_{n,d}$ and a proper norm form $H\in\cH_{n,d}$.
The intersecting property says that $H+Q-P$ has a zero whenever
$P\in\cF_A$ and $Q\in\cF_{A+H}$. By Lemma~\ref{lem:difference-of-stars},
the differences $Q-P$ run through all $R\in\cP_{n,d-1}$ with $R(\mathbf s)=b_{A+H}-b_A$. If $b_{A+H}-b_A\neq-H(\mathbf s)$, then
Lemma~\ref{lem:nonvanishing-completion}, applied at $\mathbf s$ with
$\delta=b_{A+H}-b_A$, would produce such an $R$ for which $H+R$ is
nonvanishing, a contradiction. Hence $b_{A+H}-b_A=-H(\mathbf s)$. 

Define $\xi(A)=b_A+A(\mathbf s)$ for each $A\in\cH_{n,d}$. Then $\xi(A+H)=\xi(A)$ for every proper norm form $H$. By Proposition~\ref{prop:proper-norm-generation}, the proper norm forms
span $\cH_{n,d}$ and are closed under multiplication by nonzero
scalars. Lemma~\ref{lem:spanning-invariance} applied to $\xi(A+H)=\xi(A)$ implies that $\xi$ is constant. Every member of $\cF$ can be expressed as $A+P$ for some $P\in\cF_{A}$; we compute
\[
  (A+P)(\mathbf s)=A(\mathbf s)+b_A=\xi(A)=\xi(0).
\]
This shows $\cF\subseteq \{f\in\cP_{n,d}:f(\mathbf s)=\xi(0)\}$. By comparing cardinalities, we conclude that $\cF=\{f\in\cP_{n,d}:f(\mathbf s)=\xi(0)\}$ is a star.
\end{proof}

\begin{prop}[Variable lifting]
\label{prop:variable-lifting}
Let $n\geq2$ and $d\geq1$. Suppose that every maximum intersecting
family in $\cP_{n-1,d}$ is a star and that $\cP_{n,d}$ is spanned by
nonvanishing polynomials. Let $\cF\subseteq\cP_{n,d}$ be a maximum intersecting family. After
the normalization in Corollary~\ref{cor:standard-normalization}, write
\[
  \cF=\{f+\varphi(f):f\in V_0\},
  \qquad
  V_0=\{f\in\cP_{n,d}:f(0)=0\}.
\]
Then $\varphi(f)$ depends only on the coefficients of the monomials
in $f$ involving all $n$ variables. 
In particular, if $d<n$, then $\cF$ is a star.
\end{prop}

\begin{proof}
For $i\in[n]$, let $U_i$ be the space of polynomials that do not
involve $x_i$. Then $U_i\cong\cP_{n-1,d}$, so every maximum
intersecting family in $U_i$ is a star. Since $\cP_{n,d}$ is spanned
by nonvanishing polynomials,
Lemma~\ref{lem:star-slice-synchronization} applies with $V=\cP_{n,d}$ and $U=U_i$. The intersection $\cF\cap U_i$ contains $0$ and every $x_j$ with
$j\neq i$, so its in-center is the origin. Hence
\[
  \varphi(f+u)=\varphi(f), \qquad \forall f\in V_0,\ u\in U_i\cap V_0.
\]

Thus, changing the coefficient of any nonconstant monomial that omits
at least one variable does not change $\varphi$. Therefore
$\varphi(f)$ depends only on the coefficients of the monomials
involving all $n$ variables. Finally, if $d<n$, there are no such nonconstant monomials of degree at most $d$ and thus $\cF=\{f\in\cP_{n,d}:f(0)=0\}$.
\end{proof}

\subsection{Odd characteristic}\label{subsec:odd}
The following theorem, due to Adriaensen~\cite{A22}, settles the quadratic case in one variable; for sufficiently large $q$, it
also follows from the stability theorem of Aguglia, Csajb\'{o}k, and Weiner~\cite{ACW24}. 

\begin{thm}[\cite{A22}]\label{thm:univariate}
Let $q\geq3$. Every maximum intersecting family in $\cP_{1,2}$ is a star. Equivalently, suppose that
\[
\cF=\{sz^2+tz+\eta(s,t):s,t\in \F_q\}
\]
is intersecting. Then there exist $r,c\in \F_q$ such that
\[
\eta(s,t)=c-sr^2-tr
\]
for all $s,t\in \F_q$; in particular, $\cF$ is the star with center $(r,c)$.
\end{thm}

We now extend Theorem~\ref{thm:univariate} to multivariate quadratics over fields of odd order. The key is to use induction together with Proposition~\ref{prop:variable-lifting} to reduce the number of variables.

\begin{prop}\label{prop:quadratic}
Let $q$ be odd and $n\geq1$. Every maximum intersecting family in
$\cP_{n,2}$ is a star.
\end{prop}

\begin{proof}
We argue by induction on $n$. The case $n=1$ is
Theorem~\ref{thm:univariate}. Suppose that $n\geq2$, and let
$\cF\subseteq\cP_{n,2}$ be a maximum intersecting family.
By induction and Theorem~\ref{thm:span} (3),
Proposition~\ref{prop:variable-lifting} applies. If $n\geq 3$, then
Proposition~\ref{prop:variable-lifting} already implies that $\cF$ is a star.

It remains to consider $n=2$. By
Corollary~\ref{cor:standard-normalization}, we may write
\[
\cF=\{f+\varphi(f):f\in V_0\},
\qquad
V_0=\{f\in\cP_{2,2}:f(0)=0\},
\]
where $\varphi(0)=\varphi(x_1)=\varphi(x_2)=0$.
By Proposition~\ref{prop:variable-lifting}, $\varphi(f)$ depends only on the coefficient of $x_1x_2$. Thus, there is a function
$\Phi\colon\F_q\to\F_q$, with $\Phi(0)=0$, such that $\varphi(f)=\Phi([x_1x_2]f)$. 

Fix $a,z\in\F_q$. We claim that $\Phi(z+a)=\Phi(z)$. Otherwise, $\delta=\Phi(z+a)-\Phi(z)\in\F_q^\times$. The quadratic form $H=\left(x_1+\frac a2x_2\right)^2=\Nm_{\F_{q^2}/\F_q}\left(x_1+\frac a2x_2\right)$ is a proper norm form and satisfies $[x_1x_2]H=a$. Lemma~\ref{lem:nonvanishing-completion}, applied at the origin,
gives $R\in\cP_{2,1}$ such that $N=H+R$ is nonvanishing and $N(0)=R(0)=\delta$. Choose $f\in V_0$ with $[x_1x_2]f=z$. Then the two members
\[
\begin{aligned}
f+\varphi(f)
  &=f+\Phi(z),\\
f+N-\delta+\varphi(f+N-\delta)
  &=f+N-\delta+\Phi(z+a)
\end{aligned}
\]
of $\cF$ would differ by the nonvanishing polynomial $N$, a contradiction.

Therefore, $\Phi(z+a)=\Phi(z)$. Since $a$ and $z$ were arbitrary, $\Phi$ is constant. As
$\Phi(0)=0$, we have $\Phi=0$. It follows that $\cF=\{f\in\cP_{2,2}:f(0)=0\}$ is a star.
\end{proof}
\begin{thm}
\label{thm:odd-pairwise}
Let $q$ be odd, $n\geq1$, and $d\geq2$. Then every maximum intersecting family in $\cP_{n,d}$ is a star.
\end{thm}

\begin{proof}
The base case $d=2$ is Proposition~\ref{prop:quadratic}. Since $q$ is
odd, Proposition~\ref{prop:degree-lifting} applies successively at
every larger degree.
\end{proof}

\subsection{Even characteristic: \texorpdfstring{$d>n$}{d > n}}
\label{subsec:even-supercritical}

Assume that $q\geq 4$ is even. By
Proposition~\ref{prop:degree-lifting}, it suffices to prove the
base case $d=n+1$. We use induction on $n$ by considering polynomials that omit one variable. The key ingredients are norm forms and Moore determinants.

\begin{prop}[Base case for even $q$]\label{prop:critical-even}
Let $q\geq4$ be even and let $n\geq1$. Every maximum intersecting
family in $\cP_{n,n+1}$ is a star.
\end{prop}

\begin{proof}
We proceed by induction on $n$. When $n=1$, this is
Theorem~\ref{thm:univariate}. Suppose $n\geq2$. By induction, every
maximum intersecting family in $\cP_{n-1,n}$ is a star, and
Proposition~\ref{prop:degree-lifting} gives the same conclusion in
$\cP_{n-1,n+1}$.

Let $\cF\subseteq\cP_{n,n+1}$ be a maximum intersecting family. By
Corollary~\ref{cor:standard-normalization}, we may assume that $ 0,x_1,\ldots,x_n\in\cF$. 
Put
\[
  V_0=\{f\in\cP_{n,n+1}:f(0)=0\}.
\]
The same corollary gives a function $\varphi\colon V_0\to\F_q$ such that
\begin{equation}\label{eq:critical-selector}
  \cF=\{f+\varphi(f):f\in V_0\},
  \qquad \varphi(0)=\varphi(x_1)=\cdots=\varphi(x_n)=0.
\end{equation}

\smallskip
\noindent\emph{Step 1: Reduction to monomials involving all variables.} Since $\cP_{n,n+1}$ is spanned by nonvanishing polynomials by
Lemma~\ref{lem:affine-basis-norm-span},
Proposition~\ref{prop:variable-lifting} shows that $\varphi(f)$ depends only
on the coefficients of the monomials involving all $n$ variables. The only such monomials of degree at most $n+1$ are
\[
  m_0=x_1\cdots x_n,
  \qquad
  m_i=x_i^2\prod_{j\neq i}x_j\quad(1\leq i\leq n).
\]
Define the surjective linear map
\[
\rho:\cP_{n,n+1}\to\F_q^{n+1},
  \qquad
  f \mapsto ([m_0]f,[m_1]f,\ldots,[m_n]f).
\]
Its restriction to $V_0$ is also surjective.
Thus, there is a function $\Phi\colon\F_q^{n+1}\to\F_q$, with
$\Phi(0)=0$, such that
\begin{equation}\label{eq:critical-residual-function}
  \varphi(f)=\Phi(\rho(f)), \qquad \forall f\in V_0.
\end{equation}

For a nonvanishing polynomial $H\in\cP_{n,n+1}$, we claim that
\begin{equation}\label{eq:critical-forbidden-value}
  \Phi(\mathbf z+\rho(H))-\Phi(\mathbf z)\neq H(0)
\end{equation}
for all $\mathbf z\in\F_q^{n+1}$.  Indeed, suppose to the contrary that $\Phi(\mathbf z+\rho(H))-\Phi(\mathbf z)=H(0)$. Define $h=H-H(0)\in V_0$ and choose $f\in V_0$ with $\rho(f)=\mathbf{z}$. By
\eqref{eq:critical-residual-function}, the two members $f+\varphi(f)=f+\Phi(\mathbf{z})$ and $f+h+\varphi(f+h)=f+H-H(0)+\Phi(\mathbf{z}+\rho(H))$ of $\cF$ would differ by
\[
  H-H(0)+\Phi(\mathbf z+\rho(H))-\Phi(\mathbf z)=H,
\]
contradicting the nonvanishing assumption on $H$.

\smallskip

\noindent\emph{Step 2: Using norm forms to control $\Phi$.}
Choose linearly independent
$\theta_1,\ldots,\theta_n\in\F_{q^{n+1}}$, write
$\boldsymbol\theta=(\theta_1,\ldots,\theta_n)$, and for each $\gamma\in\F_{q^{n+1}}$ set
\[
  Q_{\gamma,\boldsymbol\theta}(\mathbf x)
  =
  \Nm_{\F_{q^{n+1}}/\F_q}
  (\gamma+\theta_1x_1+\cdots+\theta_nx_n).
\]
For $1\leq i\leq n$, the coefficient $B_i(\boldsymbol\theta) =[m_i]Q_{\gamma,\boldsymbol\theta}$ is independent of $\gamma$; write
\[\mathbf B(\boldsymbol\theta)
=(B_1(\boldsymbol\theta),\ldots,B_n(\boldsymbol\theta)).
\]
The coefficient of $m_0$ is
\begin{equation}\label{eq:critical-Moore}
  \Delta_{\boldsymbol\theta}(\gamma)
  =
  \det
  \begin{pmatrix}
   \gamma&\theta_1&\cdots&\theta_n\\
   \gamma^q&\theta_1^q&\cdots&\theta_n^q\\
   \vdots&\vdots&&\vdots\\
   \gamma^{q^n}&\theta_1^{q^n}&\cdots&\theta_n^{q^n}
  \end{pmatrix}.
\end{equation}
Indeed, the coefficient is the corresponding permanent, which equals
the determinant in characteristic $2$. Hence $\Delta_{\boldsymbol\theta}:\F_{q^{n+1}}\to\F_q$ is a nonzero
$\F_q$-linear map, and the Moore determinant criterion \cite{Moo96} gives
\begin{equation}\label{eq:critical-Moore-kernel}
  \ker\Delta_{\boldsymbol\theta}
  =\Span_{\F_q}\{\theta_1,\ldots,\theta_n\}.
\end{equation}

Fix $\mathbf z\in\F_q^{n+1}$ and  $t\in\F_q^\times$. Let $a\in\F_q^\times$ be arbitrary.  Lemma~\ref{lem:prescribed-norm-linear} gives $\gamma\in\F_{q^{n+1}}$ such that
\[
  \Nm_{\F_{q^{n+1}}/\F_q}(\gamma)=a,
  \qquad
  \Delta_{\boldsymbol\theta}(\gamma)=t.
\]
Since $t\neq 0$, equation \eqref{eq:critical-Moore-kernel} gives $\gamma\notin\Span_{\F_q}\{\theta_1,\ldots,\theta_n\}$, so
$Q_{\gamma,\boldsymbol\theta}$ is nonvanishing. Moreover,
\[
Q_{\gamma,\boldsymbol\theta}(0)=a, \qquad \rho(Q_{\gamma,\boldsymbol\theta})=(t,\mathbf B(\boldsymbol\theta)).
\]
Let $H=Q_{\gamma,\boldsymbol{\theta}}$. Observe that $\Phi\bigl(\mathbf z+(t,\mathbf B(\boldsymbol\theta))\bigr)-\Phi(\mathbf z)\neq H(0)=a$ by \eqref{eq:critical-forbidden-value}. Since $a\in\F_{q}^{\times}$ was arbitrary, we must have
\begin{equation}\label{eq:critical-Phi-invariance}\Phi\bigl(\mathbf z+(t,\mathbf B(\boldsymbol\theta))\bigr)=\Phi(\mathbf z).
\end{equation}

\smallskip
\noindent \emph{Step 3: Showing that $\Phi$ is constant.} Let $\cD$ be the set of all $(t,\mathbf B(\boldsymbol\theta))$ with $t\in\F_q^\times$ and
$\theta_1,\ldots,\theta_n$ linearly independent. Equation~\eqref{eq:critical-Phi-invariance} says that $\Phi$ is invariant under translation by every element of $\cD$. Note that the set $\cD$ is closed under multiplication by nonzero scalars: if $\lambda\in\F_q^\times$, choose $\xi\in\F_{q^{n+1}}^\times$ with $\Nm_{\F_{q^{n+1}}/\F_q}(\xi)=\lambda$ and replace each $\theta_i$ by $\xi\theta_i$.

Finally, we claim that $\cD$ spans $\F_q^{n+1}$. Fix $\boldsymbol\theta$ and $c\in\F_q^\times$. Choose
$t_1\in\F_q\setminus\{0,c\}$ and put $t_2=c+t_1$. Since $q\geq4$,
both $t_1$ and $t_2$ are nonzero, and we have
\[
(t_1,\mathbf B(\boldsymbol\theta))+(t_2,\mathbf B(\boldsymbol\theta))=(c,0).
\]
Thus, $\F_q\times\{0\}$ lies in the span of $\cD$, and hence so does $(0,\mathbf B(\boldsymbol\theta))$. Projecting the spanning family in Lemma~\ref{lem:affine-basis-norm-span}
onto the coefficients of $m_1,\ldots,m_n$ shows that the vectors $\mathbf B(\boldsymbol\theta)$ span $\F_q^n$. Therefore, $\cD$ spans $\F_q^{n+1}$. Lemma~\ref{lem:spanning-invariance} now implies that $\Phi$ is constant. Since $\Phi(0)=0$, we have $\Phi=0$. We conclude that $\cF = \{f\in\cP_{n,n+1}:f(0)=0\}$ is a star.
\end{proof}

\begin{thm}
\label{thm:even-supercritical}
Let $q\geq4$ be even, let $n\geq1$, and let $d>n$. Then every maximum intersecting family in $\cP_{n,d}$ is a star.
\end{thm}

\begin{proof}
The base case $d=n+1$ is Proposition~\ref{prop:critical-even}; every larger degree follows successively from
Proposition~\ref{prop:degree-lifting}.
\end{proof}

\subsection{Even characteristic: \texorpdfstring{$d\leq n$}{d <= n}}
\label{sec:even-exceptional-classification}

We now classify all maximum intersecting families when $q\geq4$ is even and $2\leq d\leq n$. Put $\cM_{n,d}=\F_q^{\binom{n}{d}}$ and define the $\F_q$-linear map
\[  
\Omega_{n,d} \colon \cP_{n,d}\to\cM_{n,d},
  \qquad f \mapsto \bigl([x_I]f\bigr)_{I\in\binom{[n]}{d}}.
\]
Write
\[
\mathcal K_{n,d}=\ker\Omega_{n,d} \subseteq \cP_{n,d}.
\]

\begin{lem}
\label{lem:even-distinct-fibers}
Let $q\geq 4$ be even and let $2\leq d \leq n$. 
If $f,g\in\cP_{n,d}$ and $\Omega_{n,d}(f)\neq\Omega_{n,d}(g)$, then $f$ and
$g$ intersect.
\end{lem}

\begin{proof}
Choose $I\in\binom{[n]}{d}$ such that $[x_I](f-g)\neq0$, and set $x_j=0$ for all $j\notin I$. The resulting polynomial has
degree at most $d$ in the $d$ variables indexed by $I$, with nonzero $x_I$-coefficient. Lemma~\ref{lem:parity} gives a zero of this
restricted polynomial, which extends to a zero of $f-g$ in $\F_q^n$.
\end{proof}

\begin{prop}
\label{prop:even-fiber-rigidity}
Let $q\geq4$ be even and let $2\leq d\leq n$. Every maximum
intersecting family in $\mathcal K_{n,d}$ is a star.
\end{prop}

\begin{proof}
Let $\cA$ be a maximum intersecting family in $\mathcal K_{n,d}$. Translation of the variables leaves the degree-$d$ homogeneous part unchanged, so $\mathcal K_{n,d}$ is translation invariant. By Corollary~\ref{cor:standard-normalization}, we may assume that $0,x_1,\ldots,x_n\in\cA$. Writing
\[
V_0=\{f\in\mathcal K_{n,d}:f(0)=0\},
\]
the same corollary gives a function
$\varphi\colon V_0\to\F_q$ such that
\[
\cA=\{f+\varphi(f):f\in V_0\},
  \qquad \varphi(0)=\varphi(x_1)=\cdots=\varphi(x_n)=0.
\]

For $S\in\binom{[n]}{d-1}$, let $U_S$ be the space of polynomials of degree at most $d$ involving only the variables $x_i$ with $i\in S$. Then $U_S\cong\cP_{d-1,d}$ and $U_S\subseteq\mathcal K_{n,d}$. By Theorem~\ref{thm:span} (2), $\mathcal K_{n,d}$ is spanned
by nonvanishing polynomials, while Theorem~\ref{thm:even-supercritical} shows that every maximum intersecting family in $U_S$ is a star. Hence Lemma~\ref{lem:star-slice-synchronization} applies. Since $\cA\cap U_S$ contains $0$ and every $x_i$ with $i\in S$, its in-center is the origin, and therefore
\[
\varphi(f+R)=\varphi(f)
  \qquad
  \forall f\in V_0,\ R\in U_S \ \text{with}\ R(0)=0.
\]
Every monomial term of $f\in V_0$ involves at most $d-1$ variables: a degree-$d$ monomial involving $d$ variables would be squarefree, and its coefficient vanishes by the definition of $\mathcal K_{n,d}$. Thus, every monomial term belongs to some $U_S$. Removing them one at a time gives 
$\varphi(f)=\varphi(0)=0$. Hence $\varphi=0$, and $\cA=\{f\in\mathcal K_{n,d}:f(0)=0\}=V_0$. 
\end{proof}

\begin{prop}
\label{prop:even-exceptional-classification}
Let $q\geq4$ be even and let $2\leq d\leq n$. For each
$\omega\in\cM_{n,d}$, choose a pair $(\mathbf s_\omega,c_\omega)\in \F_q^n\times \F_q$, and set $\mathbf{s}=(\mathbf{s}_{\omega})_{\omega\in \cM_{n,d}}$ and $c=(c_{\omega})_{\omega\in \cM_{n,d}}$. Then
\begin{equation}\label{eq:even-exceptional-family}
  \cF_{\mathbf s,c}
  =\bigcup_{\omega\in\cM_{n,d}}
    \{f\in\cP_{n,d}:\Omega_{n,d}(f)=\omega,
       \ f(\mathbf s_\omega)=c_\omega\}
\end{equation}
is a maximum intersecting family. Conversely, every maximum
intersecting family in $\cP_{n,d}$ has a unique representation of the
form \eqref{eq:even-exceptional-family}.
\end{prop}

\begin{proof}
First consider a family of the form
\eqref{eq:even-exceptional-family}. Two members in the same
$\Omega_{n,d}$-fiber agree at the point chosen for that fiber, while two
members in different fibers intersect by
Lemma~\ref{lem:even-distinct-fibers}. Thus the family is intersecting.
The map $\Omega_{n,d}$ is surjective. Since $\mathcal K_{n,d}$ contains
the constants, the evaluation condition selects exactly
$|\mathcal K_{n,d}|/q$ polynomials from each fiber. Therefore
\[
  |\cF_{\mathbf s,c}|
  =|\cM_{n,d}|\frac{|\mathcal K_{n,d}|}{q}
  =\frac{|\cP_{n,d}|}{q},
\]
so the family is maximum.

Conversely, let $\cF\subseteq\cP_{n,d}$ be a maximum intersecting family, and
put
\[
  \cF(\omega)=\{f\in\cF:\Omega_{n,d}(f)=\omega\},  \qquad \forall \omega\in\cM_{n,d}.
\]
By Lemma~\ref{lem:transversal},
every $\cF(\omega)$ has size
$|\mathcal K_{n,d}|/q$. Choose
$g_\omega\in\Omega_{n,d}^{-1}(\omega)$. Applying
Proposition~\ref{prop:even-fiber-rigidity} to
$\cF(\omega)-g_\omega$ and translating back gives
$\mathbf s_\omega\in \F_q^n$ and $c_\omega\in \F_q$ such that
\[
  \cF(\omega)
  =\{f\in\cP_{n,d}:\Omega_{n,d}(f)=\omega,
       \ f(\mathbf s_\omega)=c_\omega\}.
\]
Taking the union over $\omega$ gives
\eqref{eq:even-exceptional-family}.

It remains to prove uniqueness of the representation. Fix a fiber $\Omega_{n,d}^{-1}(\omega)$ and suppose that
\[
  \{f\in \Omega_{n,d}^{-1}(\omega):f(\mathbf{a})=b\}
  =\{f\in \Omega_{n,d}^{-1}(\omega):f(\mathbf{a}')=b'\}.
\]
Suppose $\mathbf{a}\neq \mathbf{a}'$. Choose $h\in\cP_{n,d}$ with $\Omega_{n,d}(h)=\omega$. Consider the element $f=g+h$ where $g\in\cP_{n,1}$ such that $g(\mathbf{a})=b-h(\mathbf{a})$ and $g(\mathbf{a}') = b'-h(\mathbf{a}')+1$. Observe that $f\in \Omega_{n,d}^{-1}(\omega)$ satisfies $f(\mathbf{a})=b$ but $f(\mathbf{a}')=b'+1$, a contradiction. 

Thus, $\mathbf{a}=\mathbf{a}'$ and $b=b'$. This proves the uniqueness of the
pair $(\mathbf{a},b)$ on each fiber $\Omega_{n,d}^{-1}(\omega)$.
\end{proof}

\begin{cor}\label{cor:even-exceptional-stars}
A family in \eqref{eq:even-exceptional-family} is a star if and only if the map $\omega\mapsto(\mathbf s_\omega,c_\omega)$ is constant.
\end{cor}

\begin{proof}
A constant choice gives a star. Conversely, if the family is the star
$\{f\in\cP_{n,d}:f(\mathbf s)=c\}$, then its intersection with every
$\Omega_{n,d}$-fiber is described by the same pair $(\mathbf s,c)$.
The uniqueness in Proposition~\ref{prop:even-exceptional-classification} forces
$(\mathbf s_\omega,c_\omega)=(\mathbf s,c)$ for every $\omega$.
\end{proof}

\begin{thm}
\label{thm:even-exceptional}
Let $q\geq4$ be even and let $2\leq d\leq n$. Then
$\cP_{n,d}$ does not have the strict EKR property.
\end{thm}

\begin{proof}
Since $|\cM_{n,d}|>1$, choose a nonconstant function 
$\omega\mapsto (\mathbf s_\omega,c_\omega)$. Proposition~\ref{prop:even-exceptional-classification}
gives a maximum intersecting family, and
Corollary~\ref{cor:even-exceptional-stars} shows that it is not a star.
\end{proof}

\begin{rem}\label{rem:even-exceptional-count}
By Proposition~\ref{prop:even-exceptional-classification}, there are $q^{(n+1)q^{\binom{n}{d}}}$ maximum intersecting families in this setting, compared to only $q^{n+1}$ stars.
\end{rem}

Now we complete the proof of Theorem~\ref{thm:main}(1).

\begin{proof}[Proof of Theorem~\ref{thm:main}(1)]
Proposition~\ref{prop:degree-one} settles
$d=1$. Next assume $d\geq2$. If $q=2$, the classification is
Theorem~\ref{thm:F2-pairwise}. If $q$ is odd, it is Theorem~\ref{thm:odd-pairwise}. Finally, if $q\geq4$ is even, Theorem~\ref{thm:even-supercritical} gives the star conclusion for $d>n$, while Theorem~\ref{thm:even-exceptional} gives maximum non-star families for $2\leq d\leq n$. 
\end{proof}

\section{Zero-detecting functionals}\label{sec:GKZ}

We finish the paper by establishing a GKZ-type theorem over finite fields.

\begin{proof}[Proof of Theorem~\ref{thm:zero-detecting}]
For the first implication, suppose that $\Psi$ is zero-detecting.

\emph{Part 1.} Assume that $q\geq 3$, and put $\cA=\ker\Psi$. Since $\Psi(1)=1$, the functional $\Psi$ is nonzero, so $|\cA|=\frac{|\cP_{n,d}|}{q}$. If $f,g\in\cA$, then $f-g\in\cA$ and $f-g$ has a zero by our hypothesis on $\Psi$. Thus, $\cA$ is a maximum intersecting family.

If $q$ is odd or $d>n$, Theorem~\ref{thm:main} shows that $\cA$ is a
star. Since $0\in\cA$, there is $\mathbf a\in\F_q^n$ such that
\[
  \cA=\{f\in\cP_{n,d}:f(\mathbf a)=0\}.
\]
The functionals $\Psi$ and $\operatorname{ev}_{\mathbf a}$ have the
same kernel. It follows that they are proportional, and their common value $1$ at the constant polynomial $1$ shows that $ \Psi=\operatorname{ev}_{\mathbf a}$.

\emph{Part 2.} Now suppose that $q\geq4$ is even and $d\leq n$, and put $\mathcal K_{n,d}=\ker\Omega_{n,d}$. Since $1\in\mathcal K_{n,d}$, the restriction $\Psi|_{\mathcal K_{n,d}}$ is nonzero. Its kernel is a maximum intersecting family in $\mathcal K_{n,d}$, so Proposition~\ref{prop:even-fiber-rigidity} gives $\mathbf a\in\F_q^n$ such that
\[
  \ker(\Psi|_{\mathcal K_{n,d}})
  =\{f\in\mathcal K_{n,d}:f(\mathbf a)=0\}.
\]
The two functionals $\Psi|_{\mathcal K_{n,d}}$ and  $\operatorname{ev}_{\mathbf a}|_{\mathcal K_{n,d}}$ have the same kernel and both take the value $1$ at the constant
polynomial $1$. Hence $\Psi|_{\mathcal K_{n,d}}=\operatorname{ev}_{\mathbf a}|_{\mathcal K_{n,d}}$. For every $f\in\cP_{n,d}$,
\[
f-\sum_{I\in\binom{[n]}d}([x_I]f)\,x_I
\in\ker\Omega_{n,d}.
\]
Since
$\Psi-\operatorname{ev}_{\mathbf a}$ vanishes on this kernel, setting
$\mu_I=\Psi(x_I)-x_I(\mathbf a)$ gives
\[
  \Psi(f)
  =f(\mathbf a)
   +\sum_{I\in\binom{[n]}{d}}\mu_I\cdot[x_I]f.
\]

\emph{Part 3.} Let $q=2$, and put $\cA=\ker\Psi$. Since $\Psi(1)=1$,
the family $\cA$ has size $|\cP_{n,d}|/2$, and the zero-detecting
property makes it intersecting, as in Part~1. Proposition~\ref{prop:F2-pairwise}
shows that either $\Lambda_{0}\subseteq \cA$ or $\Lambda_{1}\subseteq \cA$. Since $\ker(\Psi)\cap \Lambda_{1}=\emptyset$, the only choice is $\Lambda_{0}\subseteq \cA$.

If two polynomials $f$ and $g$ agree at every point of $\F_2^n$, then $f-g\in \Lambda_{0}$, so $\Psi(f)=\Psi(g)$. Consequently, $\Psi$ descends to a linear functional on the space of value vectors $\bigl(f(\mathbf a)\bigr)_{\mathbf a\in\F_2^n}$ arising from polynomials. Choose a basis of this space, extend it to a basis of all $\F_2$-valued vectors indexed by $\F_2^n$, and declare the functional to be zero on the added basis vectors. The resulting functional is a linear combination of the coordinate maps, so there are coefficients $c_{\mathbf a}\in\F_2$ such that
\[
\Psi(f)=\sum_{\mathbf a\in\F_2^n}c_{\mathbf a}f(\mathbf a).
\]
Taking $S=\{\mathbf a\in\F_2^n:c_{\mathbf a}=1\}$ gives $\Psi(f)=\sum_{\mathbf a\in S}f(\mathbf a)$.  Evaluating at the constant polynomial $1$ gives $1=\Psi(1)=|S|\pmod2$, so $|S|$ is odd.

\smallskip

For the converse, a point evaluation is plainly zero-detecting. Suppose that $\Psi$ has the form in~{\rm(2)} and that $\Psi(f)=0$. If $\Omega_{n,d}(f)=0$, then $f(\mathbf a)=0$. If $\Omega_{n,d}(f)\neq0$, Lemma~\ref{lem:even-distinct-fibers}, applied to
$f$ and the zero polynomial, gives a zero of $f$.

Suppose instead that $\Psi$ has the form in~{\rm(3)}, where $|S|$ is odd. If $f$ had no zero in $\F_2^n$, then $f(\mathbf a)=1$ for every $\mathbf a\in\F_2^n$, and hence $\Psi(f)=\sum_{\mathbf a\in S}1=1$. Thus, $\Psi(f)=0$ forces $f$ to have a zero.
\end{proof}

\section*{Acknowledgments}
The authors thank Santa Clara University and E\"otv\"os Lor\'and University for their hospitality; much of this work was carried out during visits to these institutions. We thank Karen Meagher for helpful discussions. The research of the third author was supported in
part by an NSERC fellowship. 

\section*{AI disclosure}

During the preparation of this paper, the authors used AI tools interactively as an aid in discussing, refining, and presenting parts of the manuscript. The mathematical ideas, underlying intuition, and overall direction are due to the authors. The authors are solely responsible for the final content of the paper.

\bibliographystyle{abbrv}
\bibliography{main}

\end{document}